\documentclass[toc=bib, 11pt]{article}

\usepackage{amsmath, amsthm, amssymb}
\usepackage[colorlinks=true,linkcolor=blue,urlcolor=blue,citecolor=blue,unicode=true]{hyperref}
\usepackage{geometry}
\usepackage[T1]{fontenc}
\usepackage{lmodern}

\usepackage[headsepline]{scrlayer-scrpage} 
\clearmainofpairofpagestyles
\ohead{\headmark} 
\automark[section]{section} 

\usepackage{dsfont}
\usepackage{epsfig}
\usepackage{faktor}
\usepackage{extarrows}
\usepackage{graphicx}
\usepackage{graphics}
\usepackage{pifont}
\usepackage{float}
\usepackage{subfigure}
\usepackage{multirow}
\usepackage{color}
\usepackage{lineno}
\usepackage[normalem]{ulem} 
\usepackage{makeidx}
\usepackage{xspace}
\usepackage{wrapfig}
\usepackage{todonotes}
\usepackage{hyperref}
\usepackage{cleveref}
\usepackage{pst-node}
\usepackage{tikz-cd} 
\usepackage[thinlines]{easytable}

\newtheorem{Definition}{Definition}[section]
\newtheorem{theorem}[Definition]{Theorem}
\newtheorem{Corollary}[Definition]{Corollary}
\newtheorem{Theorem}[Definition]{Theorem}
\newtheorem{Lemma}[Definition]{Lemma}

\newtheorem{Proposition}[Definition]{Proposition}
\newtheorem{Observation}[Definition]{Observation}
\newtheorem{Remark}[Definition]{Remark}

\newtheorem{thm}{Theorem}

\newcommand{\kreis}[1]{\unitlength1ex\begin{picture}(2.5,2.5)%
\put(0.75,0.75){\circle{2.5}}\put(0.75,0.75){\makebox(0,0){#1}}\end{picture}}

\newcommand{\C}{\mathbb{C}} 
\newcommand{\cp}{\mathbb{C}\mathrm{P}} 
\newcommand{\ch}{\mathbb{C}\mathrm{H}} 
 
\newcommand{\R}{\mathbb{R}} 
\newcommand{\Q}{\mathbb{Q}} 
\newcommand{\Z}{\mathbb{Z}} 
 
\newcommand{\g}{\mathfrak{g}} 

\newcommand{\HH}{\mathcal{H}}
\newcommand{\VV}{\mathcal{V}}

\newcommand{\dd}{\mathrm{d}}

\newcommand{\diff}{\frac{\mathrm{d}}{\mathrm{d}t}\Big|_{t=0}}

\begin{document}
\title{On symplectic geometry of tangent bundles of\\ Hermitian symmetric spaces}
 
\author{{Johanna Bimmermann\footnote{Fakultät für Mathematik, Ruhr-Universität Bochum, Universitätsstraße 150, 44801 Bochum, Germany, johanna.bimmermann@rub.de} }}

\maketitle

\begin{abstract}
\noindent
We explicitly construct a symplectomorphism that relates magnetic twists to the invariant hyperkähler structure of the tangent bundle of a Hermitian symmetric space. This symplectomorphism reveals foliations by (pseudo-) holomorphic planes, predicted by vanishing of symplectic homology. Furthermore, in the spirit of Weinstein’s tubular neighborhood theorem, we extend the (Lagrangian) diagonal embedding of a compact Hermitian symmetric space to an open dense embedding of a specified neighborhood of the zero section. Using this embedding, we compute the Gromov width and Hofer–Zehnder capacity of these neighborhoods of the zero section.
\end{abstract}

\tableofcontents

\pagestyle{scrheadings}


\section{Introduction}
A classical example of a symplectic manifold is the cotangent bundle of a smooth manifold, $(T^*M,\dd\lambda)$ where $\lambda_{(x,p)}:=p\circ\dd\pi_{(x,p)}: T_{(x,p)}T^*M\to\R$ denotes the canonical 1-form. Given a Riemannian metric $g$ on $M$, we can use the metric isomorphism $TM\cong T^*M$ to pull $\lambda$ back to $TM$. Henceforth, we will work on the tangent bundle, and by a slight abuse of notation, we will also denote the pullback of the canonic 1-form by $\lambda$. The kinetic Hamiltonian $E:TM\to \R$, defined by $ E(x,v)=\frac{1}{2}g_x(v,v)$, generates the geodesic flow. For any closed 2-form $\sigma\in \Omega^2(M)$, the form
$$
\omega_\sigma:=\dd\lambda+\pi^*\sigma
$$
is a symplectic form on $TM$. Studying the flow induced by the kinetic Hamiltonian $E$, one finds that the flow lines are bent according to $\sigma$. This is why $\omega_\sigma$ is referred to as magnetically twisted, and the Hamiltonian flow induces by $E$ is called the magnetic geodesic flow.\\
Note that, as in the untwisted case, the projection $TM\to M$ induces a Lagrangian foliation. This makes it harder to see or understand foliations by pseudoholomorphic curves, which are crucial for defining and computing symplectic invariants such as Floer homologies or symplectic capacities. \\
Conversely, Albers--Frauenfelder--Oancea \cite[Thm.\ 3]{AFO17} and Groman--Merry \cite[Thm.\ 1.1 \& Thm.\ 1.2]{GM18} showed, using symplectic homology with twisted coefficients, that in some cases such foliations must exist. In this paper, we will construct three symplectomorphisms that reveal these foliations when $(M,g)$ is a Hermitian symmetric space with invariant Kähler form $\sigma$.\\
\noindent
Additionally, we uncover an interesting relation to the invariant hyperkähler structure present in a neighborhood of the zero-section, which was explicitly characterized by Biquard and Gauduchon \cite[Thm.\ 1]{BG21}. As a preliminary step and to illustrate the idea, we begin with the linear case $M=\C^n$. 
\\
\ \\
\noindent
\textbf{The linear picture.} We look at $M=\C^n$ with $g$ the standard Euclidean metric and $i$ the standard complex structure. The corresponding standard symplectic structure is $\sigma=g \circ (i\times 1)$. The tangent bundle $TM$ is identified with $\C^n\times \C^n$, where we can think of the factors as horizontal and vertical. Now a hyperkähler structure on $TM\cong \C^n\times\C^n$ is given by three complex structures:
$$
I:=\begin{pmatrix}
        0 & i\\
        i  & 0 \\
    \end{pmatrix},\ \
J:=\begin{pmatrix}
        0 & 1\\
        -1  & 0 \\
    \end{pmatrix},\ \
K:=\begin{pmatrix}
        i & 0\\
        0  & -i \\
    \end{pmatrix},
$$
which are compatible with the metric:
$$
G:=\begin{pmatrix}
        g & 0\\
        0  & g \\
    \end{pmatrix}.
$$
In this case, it is straightforward to see that:
$$
\omega_I=i^*\dd\lambda=\begin{pmatrix}
        0 & \sigma\\
        \sigma  & 0 \\
    \end{pmatrix},\ \ \omega_J=\dd\lambda=\begin{pmatrix}
        0 & -g\\
        g  & 0 \\
    \end{pmatrix},\ \ \omega_K=\begin{pmatrix}
        \sigma & 0\\
        0  & -\sigma \\
    \end{pmatrix}.
$$
Hence, the magnetically twisted symplectic form reads:
$$
\omega_\sigma:=\dd\lambda+\pi^*\sigma=\begin{pmatrix}
        \sigma & -g\\
        g  & 0 \\
    \end{pmatrix}.
$$
For both symplectic structures $\omega_\sigma$ and $\omega_K$ the kinetic Hamiltonian $E(x,v)=\frac{1}{2}g(v,v)$ generates a circle action. In the first case, orbits are of the form $(\gamma(t), \dot\gamma(t))$, where $\gamma$ parametrizes a circle of radius $\vert v\vert$ in the affine plane $\gamma(0)+\mathrm{span}\lbrace\dot\gamma(0),i\dot\gamma(0)\rbrace$. In the second case, the orbits are simply given by $(x, e^{it}v)$. The map that intertwines these circle actions:
$$
(\C^n\times\C^n, \omega_\sigma)\to (\C^n\times\C^n, \omega_K);\  \ \begin{pmatrix}x\\v\end{pmatrix}\mapsto \begin{pmatrix}x+iv\\v\end{pmatrix}=\begin{pmatrix}
        1 & i\\
        0  & 1\\
    \end{pmatrix}\begin{pmatrix}x\\v\end{pmatrix}
$$
is a symplectomorphism, since
$$
\begin{pmatrix}
        1 & 0\\
        -i  & 1 \\
    \end{pmatrix}
    \begin{pmatrix}
        \sigma & -g\\
        g  & 0 \\
    \end{pmatrix}
    \begin{pmatrix}
        1 & i\\
        0  & 1 \\
    \end{pmatrix}=
    \begin{pmatrix}
        \sigma & 0\\
        0  & -\sigma \\
    \end{pmatrix}.
$$
\noindent
\textbf{Complex projective and complex hyperbolic space.} It was shown by Benedetti--Ritter \cite[Appendix A]{BR19} that one can generalize this map to the case $M=\cp^1$ with standard Fubini--Study Kähler form $\sigma=\omega_{FS}$. Indeed, magnetic geodesics are geodesic circles of radius $R=\tan^{-1}(\vert v\vert)$ and the map
$$
(T\cp^1, \omega_\sigma)\to (T\cp^1, \omega_K);\ \  (x,v)\mapsto\left(\gamma(1), \left(P_\gamma v\right)(1)\right)
$$
is a symplectomorphism. Here, $P$ denotes the parallel transport with respect to the Levi-Civita connection along the geodesic 
$$
\gamma: [0,1]\to\cp^1;\ \ \gamma(t):=\exp_x\left(-tb(\vert v\vert^2)iv\right)
$$
connecting $x$ and the center of the geodesic circle. Therefore, 
\begin{equation}\label{b}
    b: \R\to\R;\ b(y):=\frac{\tan^{-1}(\sqrt{y})}{\sqrt y}.
\end{equation}
By replacing $\tan$ with $\tanh$, the exact same map yields a symplectomorphism 
$$
(D_1\ch^1, \omega_\sigma)\to (D_1\ch^1, \omega_K).
$$
The fact that for the hyperbolic plane the symplectomorphism exists only on the unit disc bundles is due to the fact that magnetic geodesics are only periodic below Ma\~{n}\'{e}'s critical value, which in this case is $c(\ch^1,g,\sigma)=\frac{1}{2}$, see \cite[Sec.\ 5.2]{CFP10}. Additionally, as discovered by Calabi \cite{C79}, the maximal neighborhood of the zero section on which the invariant hyperkähler structure, particularly the symplectic form $\omega_K$, exists is also $D_1M$.

\begin{Remark}
    The symplectomorphism $(D_1\ch^1, \omega_\sigma)\to (D_1\ch^1, \omega_K)$ is equivariant with respect to the induced action by isometries of $\ch^1$. This implies that the map descends to the quotient $D_1(\ch^1/\Gamma)$ for any torsion-free discrete group of isometries $\Gamma$. Hence, the symplectomorphism exists for any hyperbolic surface.
\end{Remark}

\noindent
The symplectic form $\omega_K$ turns the standard fibration $TM\to M$ into a symplectic fibration. However, unlike in the linear case, $\omega_K$ is not constant along the fiber. Benedetti and Ritter \cite[Appendix A]{BR19} found a scaling along the fibers that addresses this issue:
$$
(T\cp^1, \omega_K)\to (T\cp^1, \dd\tau/2+\pi^*\sigma);\ \  (x,v)\mapsto\left(x, e^{2a(\vert v\vert^2)}v\right).
$$
Here, $\tau$ denotes the connection 1-form viewing $T\cp^1$ as principal $U(1)$-bundle, and 
\begin{equation}\label{a}
    a:\R\to\R;\ a(y):=\frac{1}{2}\ln \left( \frac{2}{y}\left(\sqrt{1+y}-1\right)\right).
\end{equation}
Again, replacing $\vert v\vert^2$ with  $ -\vert v\vert^2$, yields a symplectomorphism 
$$
(D_1\ch^1, \omega_K)\to (D_2\ch^1, \dd\tau/2+\pi^*\sigma ).
$$
\begin{Observation}
   We can handle both the compact case ($\cp^1$) and the non-compact case ($\ch^1$) simultaneously by setting $y:=\kappa \vert v\vert^2$, where $\kappa=1$ for the compact case and $\kappa=-1$ for the non-compact case.
\end{Observation}
\noindent
The symplectic form $\dd\tau/2+\pi^*\sigma$ endows $T\Sigma$ for $\Sigma\in \lbrace \cp^1,\ch^1\rbrace$ with the structure of a symplectic vector bundle. In particular it is not hard to write down an almost complex structure so that the fibers are pseudoholomorphic planes. This foliation leads to the vanishing of symplectic homology in the case of $\cp^1$, as shown by Ritter \cite[Thm.\ 5]{Rit11}. Interestingly, symplectic homology is not even defined in the case of $\ch^1$ because its boundary is not convex (the Liouville vector field points inwards). \\
\ \\
\noindent
In the case of $\cp^1$, there exists another intriguing symplectomorphism that persists even when the magnetic term vanishes. It was shown in \cite[Thm.\ C]{proj23} that the map
$$
(D_1\cp^1,\omega_{s\sigma})\to(\cp^1\times\cp^1\setminus\bar\Delta,R_-\sigma\ominus R_+\sigma);\ (x,v)\mapsto \Big (\exp_x(-c_-(\vert v\vert)iv),\exp_x(c_+(\vert v\vert)iv)\Big )
$$
is a symplectomorphism. Here, $\bar\Delta$ denotes the anti-diagonal, $R_\pm=\frac{1}{2}(\sqrt{s^2+1}\pm s)$ and $c_\pm: (-1,1)\to (0,\infty)$ are the smooth, even functions \cite[Lem.\ 5.2]{proj23} implicitly defined through the equations
\begin{align}\label{c}
    &R_-\sin(2c_-(y)y)+R_+\sin(2c_+(y)y)=y\\
    &R_-\cos(2c_-(y)y)-R_+\cos(2c_+(y)y)=R_--R_+.
\end{align}
\ \\
\noindent
\textbf{Hermitian symmetric spaces.}
In this paper, we will generalizes the aforementioned symplectomorphisms to Hermitian symmetric spaces $M\cong G/K$ equipped with the up to scaling unique invariant Kähler structure $(g,j,\sigma)$. Our approach heavily relies on explicit formulas derived by Biquard--Gaudouchon \cite[Thm.\ 1]{BG21} for the invariant hyperkähler structure in the tangent bundle of Hermitian symmetric spaces. For compact type spaces, the hyperkähler structure is defined on the full tangent bundle, while for non-compact type spaces it only exists on $U_1M$, where
$$
U_\rho M:=\lbrace (x,v)\in TM\ : \ \vert g_x(jR(jv,v)w,w)\vert\leq \rho^2 \vert w\vert^2 \ \forall w\in T_xM\rbrace.
$$
Observe that in the case of the complex hyperbolic space $\ch^1$, the unit disk bundle is precisely this neighborhood where the hyperkähler structure exists.\\
\ \\
\noindent
\textbf{From symplectically twisted to hyperkähler.}
\begin{thm}\label{twisted to hyperkähler}
If $M\cong G/K $ is a Hermitian symmetric space of compact type there is a $G$-equivariant symplectomorphism identifying
$$
(T M,\omega_\sigma)\cong(T M, \omega_K).
$$
If $M$ is a Hermitian symmetric space of non-compact type, then
$$
(U_{1}M,\omega_\sigma)\cong(U_{1} M, \omega_K).
$$
\end{thm}
\noindent
The symplectic form $\omega_K$ turns the standard fibration $TM\to M$ into a symplectic fibration, i.e. $\omega_K\vert_{T_xM}$ is symplectic, but not constant along the fiber. Our second theorem states that one can find a fiberwise map, so that the symplectic form becomes fiberwise constant.\\
\ \\
\noindent
\textbf{From hyperkähler to fiberwise constant.} Let $\kappa: TTM\to \mathcal{V}\cong TM$ denote the projection to the vertical distribution with respect to the Levi-Civita connection. Define a 1-form $\tau\in\Omega^1(TM)$ pointwise as
$$
\tau_{(x,v)}(\cdot):=g_x(jv, \kappa(\cdot)).
$$
As shown in Proposition \ref{derivatives 1-forms}, $\dd\tau$ is fiberwise constant, more precisely:
$$
\left(\dd\tau_{(x,v)}\right)\vert_{\mathcal{V}\times\mathcal{V}}=2\sigma_x\ \ \forall v\in T_xM.
$$
Observe, that $\dd\tau$ cannot be non-degenerate as $\tau$ vanishes constantly along the zero-section. However, we will see in Corollary \ref{symplectic fibers} that $\dd\tau/2+\pi^*\sigma$ is non-degenerate, hence symplectic on $TM$ or $U_2M$ depending on the type of $M$.

\begin{thm}\label{hyperkähler to constant}
    If $M\cong G/K $ is a Hermitian symmetric space of compact type there is a $G$-equivariant symplectomorphism identifying
    $$
    (T M,\omega_K)\cong(T M, \dd\tau/2+\pi^*\sigma).
    $$
    If $M$ is a Hermitian symmetric space of non-compact type, then
    $$
    (U_{1}M,\omega_K)\cong(U_{2} M, \dd\tau/2+\pi^*\sigma).
    $$
\end{thm}

\noindent
\textbf{Diagonal embeddings.}
The previous symplectomorphisms exclusively work with a non-vanishing magnetic field. However, the result by Albers--Frauenfelder--Oancea \cite[Thm.\ 3]{AFO17}, indicates that these foliations by pseudoholomorphic planes should, in the compact case, also exists without a magnetic twist. This can be seen, as the diagonal embedding $M\hookrightarrow (M\times M,\sigma\ominus \sigma)$ is Lagrangian, which by Weinstein's tubular neighborhood theorem implies that also a neighborhood of the zero-section embeds. The product is uniruled \cite[Lem.\ 15]{LMZ13}. Using symmetries we can extend this embedding to an open dense embedding of an explicit neighborhood of the zero-section. We can still include a magnetic twist, but the following theorem also works without.

\begin{thm}\label{diagonal}
Let $M\cong G/K$ be an irreducible Hermitian symmetric space of compact type, then there exists a $G$-equivariant symplectomorphism
$$
\phi: (U_{2\sqrt{R}}M,\omega_{(1-R)\sigma})\rightarrow (M\times M\setminus \bar\Delta, \sigma\ominus R\sigma),
$$
where $\bar \Delta$ is fiberwise the cut locus of the base point and thus a finite union of complex submanifolds of complex codimension one.
\end{thm}

\noindent
\textbf{Computation of symplectic capacities.}
Symplectic capacities have been a fascinating area of symplectic geometry since Gromov proved his non-squeezing theorem \cite{G85}. Roughly speaking, they fall into two types: embedding capacities and dynamical capacities. The Gromov width is of the former type and is defined as follows. For any symplectic manifold $(N,\omega)$ of dimension $2n$ we look for the largest standard ball $(B^{2n}(r),\omega_0)$ of radius $r$ that symplectically embeds, i.e.
$$
c_G(N,\omega):=\sup\lbrace\pi r^2\vert (B^{2n}(r),\omega_0)\hookrightarrow (N,\omega)\rbrace.
$$
On the other hand, the Hofer--Zehnder capacity belongs to the latter type. For a symplectic manifold $(N,\omega)$ possibly with boundary $\partial N$, the Hofer--Zehnder capacity is defined as:
$$
c_{HZ}(N,\omega):=\sup\left\lbrace \mathrm{max}(H)\ \vert\ H: N\to \R\ \  \text{smooth, admissible}\right\rbrace,
$$
where admissible means:
\begin{itemize}
    \item $0\leq H$ and there exists an open set $U\subset N\setminus\partial N$ such that $H\vert_U\equiv 0$,
    \item there exists a compact set $K\subset N\setminus\partial N$ such that $H\vert_{N\setminus K}\equiv\max(H)$,
    \item all non-constant periodic solutions $\gamma:\R \to N$ of $\dot \gamma=X_H$ have period $T>1$.
\end{itemize}
 Here, $X_H$ denotes the Hamiltonian vector field defined imposing the relation $\dd H=-\iota_{X_H}\omega$.\\
 \ \\
 \noindent
 Looking at the definitions there is no reason to expect that their values have anything to do with each other. Interestingly, in many example where both can be computed, they agree. Note that this list is not particularly long, but includes for example monotone convex toric domains in $\R^{2n}$ \cite[Thm.\ 1.7]{GHR22}.  In recent years, considerable research on capacities has revolved around the Viterbo conjecture:
 \begin{quote}
     "All normalized symplectic capacities on convex domains in $\R^{2n}$ agree."
 \end{quote}
 Recently, Haim-Kislev and Ostrover \cite{HO24}  provided a counterexample comparing the Gromov width and Hofer--Zehnder capacity of a specific convex domain. This result underscores the ongoing interest and complexity in understanding why symplectic capacities either agree or disagree.\\
 \ \\
 \noindent
 When it comes to subsets of (twisted) tangent bundles even less is known. Ferreira, Ramos and Vicente \cite{FRV23,FR22} computed the Gromov width of standard disk tangent bundles of certain spheres of revolution. They find that the value is always spectral, meaning it corresponds to the length of some closed geodesic, although not necessarily the minimum and not always represented by a simple curve. The situation for the Hofer--Zehnder capacity for standard disk tangent bundle is similarly restricted, with computations available only for selected examples. The list includes real and complex projective spaces \cite{proj23}, flat tori \cite{BBZ23} and some lens spaces \cite{BM24}.
 In all cases, the capacity is spectral, yet it does not always coincide with the Gromov width \cite{BM24}. For Hermitian symmetric spaces it is also possible to consider magnetically twisted tangent bundles, this was previously done only for $\cp^n$ \cite[Sec.\ 5]{proj23} and constant curvature surfaces \cite{Bim23}. In this paper we will generalize these results to obtain the following theorems. 

\begin{thm}\label{capacompact} Let $(M,j,\sigma)$ be a irreducible Hermitian symmetric space of compact type. Normalize $\sigma$ such that evaluated on the generator of $H_2(M,Z)$ it takes value $4\pi$, then
$$
    c_{\mathrm G}(U_{2\sqrt{R}}M,\omega_{(1-R)\sigma})=c_{\mathrm{HZ}}(U_{2\sqrt{R}}M,\omega_{(1-R)\sigma})=c_{\mathrm{HZ}}^0(U_{2\sqrt{R}}M,\omega_{(1-R)\sigma})=\min \lbrace 1,R\rbrace 4\pi.
$$
\end{thm}
\noindent
We can actually deduce the value of the Hofer--Zehnder capacity for untwisted disk tangent bundles. The result fits well with the previous results \cite{BBZ23, proj23, BM24, FRV23} mentioned above.
\begin{Corollary}
If $l$ denotes the length of a shortest closed geodesic, then
$$
c_{\mathrm{HZ}}(D_{1}M,\dd\lambda)=l.
$$
\end{Corollary}
\begin{proof}
    The upper bound follows directly as $D_1M\subset U_1M$ and $c_{HZ}(U_1M,\dd \lambda)=2\pi=l$. The fact that $l=2\pi$ comes from the normalization of $\sigma$. For the lower bound consider the Hamiltonian $H:TM\setminus{0_{TM}}\to\R;\ (x,v)\mapsto l\vert v\vert_x$. The induced Hamiltonian flow is the geodesic flow parametrized so that all orbits $(\gamma(t),\dot\gamma(t))\in TM$ satisfy $\vert \dot\gamma(t)\vert=l$. The fastest periodic orbit has period one as it corresponds to the shortest closed geodesic with speed $l$. We can now modify, $H$ similar to what is done in Lemma \ref{lem9}, to extend $H$ smoothly to the zero-section and make it admissible without changing its oscillation much. Thus $l$ is also a lower bound for $c_{\mathrm{HZ}}(D_{1}M,\dd\lambda)$.
\end{proof}
\noindent
Finally, we also obtain some explicit bounds in the negative curvature case.
\begin{thm}\label{capanoncompact}
Let $(M,g,\sigma)$ be isometrically covered by an irreducible Hermitian symmetric space of non-compact type with rank $r$, then
$$
2\pi\left(s-\sqrt{s^2-1/r}\right)\leq c_{HZ}(D_1 M,\omega_{s\sigma})\leq c_{HZ}^0(D_1 M,\omega_{s\sigma})\leq 2\pi r\left (s-\sqrt{s^2- 1}\right )
$$
for any constant $s>1$.
\end{thm}
\noindent
Observe that in the rank one case (r=1) we can compute the exact value and recover the result in \cite{Bim23}.\\
\ \\
\noindent
\textbf{Outline.} In the second section \ref{prelim} we collect techniques and results about general geometry of tangent bundles, magnetic systems, coadjoint orbits and Hermitian symmetric spaces from the literature. The most important aspect for this paper is presenting any Hermitian symmetric spaces as coadjoint orbit (Cor.\ \ref{HSS as coadjoint orbit}). This yields explicit formulas for the moment maps of the different symplectic forms we study and those will be essential to show that the diffeomorphism we construct are actually symplectic. In section \ref{symplectomorphisms} we construct the symplectomorphisms of Theorem \ref{twisted to hyperkähler},\ref{hyperkähler to constant} and \ref{diagonal}. Finally, in section \ref{capacities} we compute the Gromov width and the Hofer--Zehnder of the $U$-neighborhoods, as stated in Theorem \ref{capacompact}, strongly using the diagonal embedding. We further derive the bounds on the Hofer--Zehnder capacity in Theorem \ref{capanoncompact} of disc tangent bundles and $U$-neighborhoods using the symplectomorphisms of Thm.\ \ref{twisted to hyperkähler} and Thm.\ \ref{hyperkähler to constant}.\\
\ \\
\noindent
\textbf{Acknowledgment.} Most of the results presented in this paper were achieved during my PhD studies and would not have been possible without the invaluable support of my advisor, Gabriele Benedetti. I also wish to extend my gratitude to Beatrice Pozzetti for her enlightening lecture on symmetric spaces, which facilitated my entry into this field. Additionally, I am grateful to Brayan Ferreira and Alejandro Vicente for generously sharing their expertise and introducing me to an old trick by McDuff--Polterovich, which proved useful in computing not only Hofer--Zehnder capacities but also Gromov widths. The author further acknowledges funding by the Deutsche Forschungsgemeinschaft (DFG, German Research Foundation) – 281869850 (RTG 2229), 390900948 (EXC-2181/1) and 281071066 (TRR 191).

\section{Preliminaries}\label{prelim}
\subsection{Geometry of tangent bundles}
In this section we recall the elements of \cite{GSK02} that are relevant for our work. We modify some of the calculations, as all calculations in \cite{GSK02} were done for the Levi-Civita, but work analogously for other connections. For general connections a reference is \cite{D62}. Let $(M,\sigma)$ be a symplectic manifold and denote by $\pi: TM\to M$ the tangent bundle. The kernel of the differential $\dd\pi: TTM\to TM$ defines a distribution in $\mathcal{V}\subset TTM$, called \emph{vertical distribution}. Choose a compatible almost complex structure $j$ on $M$ and denote the associated metric by $g$. The almost complex structure is in general not integrable still it turns the tangent bundle into a complex vector bundle. If we pick as connection the following modification 
$$
\Tilde{\nabla}_XY:=\nabla_XY-\frac{1}{2}j(\nabla_X j)Y
$$
of the Levi-Civita connection $\nabla$, we turn the tangent bundle into a Hermitian vector bundle. Indeed $\Tilde{\nabla}g=0$ and $\Tilde{\nabla}j=0$, but the new connection $\Tilde{\nabla}$ is not torsion free in general. Its torsion is precisely the Nijenhuis-tensor of $j$.
This connection determines a complement $\mathcal{H}\subset TTM$ called \emph{horizontal distribution}, i.e.\
$$
TTM=\mathcal{H}\oplus\mathcal{V}.
$$
Both $\mathcal{V}_{(x,v)}$ and $\mathcal{H}_{(x,v)}$ are as vector spaces isomorphic to $T_xM$. In particular we can lift vectors in $T_xM$ and also vector fields on $TM$ horizontally and vertically.
\begin{Definition}[Horizontal \& vertical lift]\label{lifts}
Let $(x,v)\in TM$. Any tangent vector $w\in T_xM$ can be lifted horizontally (vertically) to a vector in $\mathcal{H}_{(x,v)}$ ($\mathcal{V}_{(x,v)}$). Explicitly the horizontal lift is defined by
$$
w\mapsto w^\mathcal{H}:=\diff (\gamma(t),P_\gamma v)
$$
where $\gamma:(-\varepsilon,\varepsilon)\subset\R\to M$ is a smooth curve satisfying $\gamma(0)=x$ and $\dot\gamma(0)=w$ and $P_\gamma$ denotes parallel transport along $\gamma$. The vertical lift is defined by
$$
w\mapsto w^\mathcal{V}:=\diff (x, v+tw).
$$
\end{Definition}
\noindent
The following proposition gives the commutator relations for vertical and horizontal lifts of vector fields.
\begin{Proposition}[\cite{D62}, Lemma 2]\label{commutator lifts}\ \\
Let $X$ and $Y$ be vector fields on $M$, then their lifts satisfy the following commutator relations:
\begin{itemize} 
    \item[(i)] $[X^\mathcal{V},Y^\mathcal{V}]=0,$
    \item[(ii)] $[X^\mathcal{H},Y^\mathcal{V}]=(\tilde\nabla_XY)^\mathcal{V},$
    \item[(iii)] $[X^\mathcal{H},Y^\mathcal{H}]=[X,Y]^\mathcal{H}-(R(X,Y)v)^\mathcal{V}.$
\end{itemize}
Here, $R$ denotes the Riemannian curvature tensor of the Hermitian connection $\tilde \nabla$.
\end{Proposition}
\noindent
We can define four vector fields 
\begin{equation}\label{XHYV}
    X_{(x,v)}:=v^\mathcal{H}, \
    H_{(x,v)}:=(jv)^\mathcal{H}, \
    Y_{(x,v)}:=v^\mathcal{V}, \
    V_{(x,v)}:=(jv)^\mathcal{V}.
\end{equation}
\noindent
Observe that these are not lifts of vector fields on the base, but locally can be written as linear combinations of those. For example
$$
X_{(x,v)}=v^i\partial_i^{\mathcal{H}},
$$
for some local frame $(\partial_1,\ldots,\partial_{2n})$ of $TM$.
Using Proposition \ref{commutator lifts} one can compute their commutators. 
\begin{Proposition}\label{commutator XHYV} The vector fields $X,H,Y,V$ satisfy
$$
[V,X]=H,\ \ \ \ [V,H]=-X,\ \ \ \ [V,Y]=0,\ \ \ \ [Y,X]=X,\ \ \ \ [Y,H]=H,
$$
and 
$$
[X,H]_{(x,v)}=(R(v,jv)v)^\mathcal{V}.
$$
\end{Proposition}
\noindent
These vector fields are non-zero and linearly independent outside the zero-section. To obtain a dual description on the cotangent bundle in terms of dual 1-forms we need to choose a metric on $TM$. For this we pick the Sasaki-metric $\hat g$. This is the metric that takes the form $g\oplus g$ with respect to the splitting $\mathcal{H}\oplus\mathcal{V}$ in horizontal and vertical. It turns out that the one-forms dual to $X,Y,H,V$ are no strangers.
\begin{Lemma}\label{dual 1-forms}
The 1-forms that are via the Sasaki-metric dual to $X, H, Y, V$ respectively are 
\begin{itemize}
    \item $\lambda$ the (metric pullback) of the canonical 1-form on $T^*M$,
    \item $j^*\lambda=$ the pullback of $\lambda$ via $j:TM\to TM$,
    \item $\dd E$ where $E(x,v)=\frac{1}{2}g_x(v,v)$ is the kinetic energy,
    \item $\dd^c E=\dd E\circ I$ where $I=j\ominus j$ is an almost complex structure on TM.
\end{itemize}
\end{Lemma}
\begin{proof}
Clearly, 
$$
\iota_X\hat g(\cdot)=\hat g(X,\cdot)=g(v,\dd\pi\cdot)=\lambda(\cdot).
$$
Similarly,
$$
\iota_H\hat g(\cdot)=\hat g(H,\cdot)=g(jv,\dd\pi\cdot)=g(jv,\dd\pi\dd j\cdot)=\lambda(\dd j\cdot)=j^*\lambda(\cdot),
$$
where we used that $\dd\pi\dd j=\dd\pi$ as $j$ is a bundle map lifting the identity. The third one is maybe the most tricky. First observe that $\dd E$ vanishes on $\mathcal{H}$ as for every $w\in T_xM$ we find
$$
\dd E(w^{\mathcal{H}})_{(x,v)}=\diff E(\gamma(t), P_\gamma v(t))=\frac{1}{2} \diff g_{\gamma(t)}(P_\gamma v(t),P_\gamma v(t))=0
$$
because $\Tilde{\nabla}g=0$. Further 
\begin{align*}
\dd E(w^{\mathcal{V}})=\diff E(x, v+tw)=\frac{1}{2}\diff g_x(v+tw,v+tw)=g_x(v,w)
\end{align*}
and we conclude 
$
\hat{g}(Y,\cdot)=\dd E. 
$
Finally
$$
\iota_V\hat g(\cdot)=\hat{g}(V,\cdot)=g(jv,\mathcal{P}_\VV\cdot)=-g(v,j\mathcal{P}_\VV\cdot)=g(v,\mathcal{P}_\VV I\cdot)=\dd E\circ I(\cdot),
$$
where $\mathcal{P}_\VV$ denotes projection to the vertical subspace.
\end{proof}
\noindent
 We shall call the dual of $V$ the \emph{angular form}, as it is dual to the vector field that generates rotation $e^{jt}:T_xM\to T_xM$ in the fibers, and denote it by $\tau:=\iota_V\hat g=\dd^c E$. Using the commutator relations in Proposition \ref{commutator lifts} we can compute the exterior derivatives of $\lambda$ and $\tau$.

\begin{Proposition}\label{derivatives 1-forms} We write the 2-forms in matrix representation with respect to the splitting of $TTM=\HH\oplus \VV$. So the upper left entry eats two horizontal vectors, the upper right a horizontal and a vertical and so on. In this representation the exterior derivatives of $\lambda$ and $\tau$ are given as
\begin{align*}
    \dd\lambda&=\left(\begin{array}{rr} g(v,T(\cdot,\cdot)) & -g \\
      g & 0 
      \end{array}\right),\ \
    \dd\tau=\left(\begin{array}{rr} g(jv, R(\cdot,\cdot)v) & 0 \\
      0 & 2\sigma
      \end{array}\right).
\end{align*}
Here, $T$ and $R$ denote respectively torsion and curvature of $\Tilde{\nabla}$.
\end{Proposition}
\begin{proof}
We first prove two identities that will be useful. Let $A,B$ be any vector fields on $M$, then
\begin{align*}
    A^\mathcal{V}(g(v,B))(x)=\diff g_x(v+tA_x,B_x)=g_x(A_x,B_x).
\end{align*}
Further
$$
A^\mathcal{H}(g(v,B))(x)=\diff g_{x(t)}(P_x v(t),B_{x(t)})=g_x(v,(\tilde\nabla_A B)_x),
$$
where $x(t)\in M$ is an integral curve of $A$, i.e.\ $\dot x =A$ and $P_x v(t)$ denotes the parallel transport of $v\in T_{x(0)} M$ along $x(t)$ with respect to the Hermitian connection $\tilde\nabla$. In total we find
$$
A^\mathcal{V}(g(v,B))=g(A,B),\ \ \ A^\mathcal{H}(g(v,B))=g(v,\tilde\nabla_A B).
$$
Similarly also
$$
A^\mathcal{V}(g(jv,B))=g(jA,B),\ \ \ A^\mathcal{H}(g(jv,B))=g(jv,\tilde\nabla_A B)
$$
holds.
We can now compute $\dd\lambda$ using Proposition \ref{commutator lifts} and the formula for the differential of a 1-form
$$
\dd\lambda(\hat A,\hat B)=\hat A\lambda(\hat B))-\hat B(\lambda(\hat A))-\lambda([\hat A,\hat B])
$$
for any vector fields $\hat A,\hat B$ on $TM$. Now
\begin{align*}
    \dd\lambda(A^\VV,B^{\VV})&=A^\VV(\lambda(B^\VV))-B^\VV(\lambda(A^\VV))-\lambda([A^\VV,B^\VV])=0\\
    \dd\lambda(A^\VV,B^{\HH})&=A^\VV(\lambda(B^\HH))-B^\HH(\lambda(A^\VV))-\lambda([A^\VV,B^\HH])=A^\VV(g(v,B))=g(A,B)\\
    \dd\lambda(A^\HH,B^{\HH})&=A^\HH(\lambda(B^\HH))-B^\HH(\lambda(A^\HH))-\lambda([A^\HH,B^\HH])\\
    &=A^\HH(g(v,B))-B^\HH(g(v,A))-\lambda([A,B]^\HH)\\
    &=g(v,\tilde \nabla_A B)-g(v,\tilde \nabla_B A)-g(v,[A,B])=g(v,T(A,B)).
\end{align*}

\noindent
Similarly we can also compute $\dd\tau$
\begin{align*}
\dd\tau(A^\VV,B^\VV)&=A^\VV(\tau(B^\VV))-B^\VV(\tau(A^\VV))-\tau([A^\VV,B^\VV])\\
&= A^\VV(g(jv,B))-B^\VV(g(jv,A))=g(jA,B)-g(jB,A)=2\sigma(A,B)\\
\dd\tau(A^\VV,B^\HH)&=A^\VV(\tau(B^\HH))-B^\HH(\tau(A^\VV))-\tau([A^\VV,B^\HH])\\
&=-B^\HH(g(jv,A))+\tau((\tilde \nabla_B A)^\VV)=-g(jv,\tilde \nabla_B A)+g(jv,\tilde \nabla_B A)=0\\
\dd\tau(A^\HH,B^\HH)&=A^\HH(\tau(B^\HH))-B^\HH(\tau(A^\HH))-\tau([A^\HH,B^\HH])\\
&=\tau((R(A,B)v)^\VV)=g(jv,R(A,B)v).
\end{align*}
\end{proof}
\noindent
Observe that $\dd\tau$ is not symplectic as it degenerates on the zero-section. However, we can add the pullback of $\sigma$ to make it non-degenerate, at least in a neighborhood of the zero-section.

\begin{Corollary}\label{symplectic fibers}
 For any real number $s>0$, the closed two-form
\begin{equation*}
\dd\tau/2-s\pi^*\sigma
\end{equation*}
is non-degenerate, and thus symplectic, in the neighborhood of the zero-section 
\begin{equation}\label{U neighborhood}
U_{2s}M:=\lbrace (x,v)\in TM\vert\ g(jv, R(w,jw)v)\leq 2s\sigma(w,jw)=2s \Vert w\Vert^2\ \ \forall w\in T_xM\rbrace
\end{equation}
determined by the holomorphic bisectional curvature.
\end{Corollary}

\noindent
\textbf{Hyperkähler structure.}
We will now have a closer look at the tangent bundle of Kähler manifolds. Denote by $\eta:=j^*\lambda$. In view of Proposition \ref{derivatives 1-forms}, we see that on the tangent bundle of a Kähler manifolds two symplectic structures naturally arise, namely
$$
\dd\lambda\equiv\begin{pmatrix}
0 & -g \\
g & 0 
\end{pmatrix} \ \ \text{and}\ \ \dd\eta\equiv\begin{pmatrix}
0 & \sigma \\
\sigma & 0 
\end{pmatrix}. 
$$
The blocks of the matrices represent the splitting into horizontal and vertical coordinates. The torsion term vanishes for an integrable almost complex structure. One reasonable question is, do they belong to a hyperkähler structure? 
\begin{Definition}[\cite{C79} Hyperkähler Structure]\label{hyperkähler structure}\ \\
A hyperkähler manifold is a Riemannian
manifold $(N,G)$ endowed with three complex structures $I$, $J$ and $K$ compatible with $G$, i.e.\ the forms
$\omega_I(\cdot,\cdot):=G(I\cdot,\cdot)$, $\omega_J(\cdot,\cdot):=G(J\cdot,\cdot)$ and $\omega_K(\cdot,\cdot):=G(K\cdot,\cdot)$ are closed and thus symplectic. Further, $I$, $J$ and $K$, considered as endomorphisms of the real tangent bundle, satisfy the relation $I \circ J = -J \circ I = K$.
\end{Definition}
\noindent
Actually the integrability condition on the almost complex structures $I$, $J$ and $K$ in the definition is redundant by the following proposition. 
\begin{Proposition}[\cite{Ht92}, Thm.\ 2]\label{integrability redundant}
A Riemannian manifold $(N,G)$ with two almost complex structures satisfying $I\circ J=-J\circ I=:K$ is hyperkähler if and only if the corresponding forms $\omega_I$, $\omega_J$ and $\omega_K$ are closed.
\end{Proposition}
\noindent
Any hyperkähler $(M, I,J, K, G)$ manifold also admits a holomorphic symplectic structure $\omega_c:=\omega_J+i\omega_K\in\Omega^{2,0}(M,\C)$ with respect to the complex structure $I$. The converse is not true in general.\\
\ \\
\noindent
Assume now that $\dd\lambda$ and $\dd \eta$ belong to a hyperkähler structure. Then there must be a metric $G$ on $TM$ and two complex structures $J$ and $K$ such that
$$
\dd\lambda(\cdot,I\cdot)=G(\cdot,\cdot)=\dd\eta(\cdot,J\cdot).
$$
It follows that
$$
\dd\lambda(\cdot, IJ\cdot)=-\dd\eta(\cdot,\cdot),
$$
and thus the third complex structure is implicitly determined to be
$$
K=IJ\equiv \begin{pmatrix}
-j & 0\\
0 & j 
\end{pmatrix}.
$$
This is indeed an integrable complex structure! It regards the tangent bundle with a canonic holomorphic symplectic structure
$$
\omega_c:=\dd\lambda+i\dd\eta\in\Omega^{2,0}(TM,\C).
$$
\noindent
But if there is also a hyperkähler structure what could $G$ be? A first guess might be the Sasaki-metric, but sadly with this choice it turns out that 
$\omega_K=G (K\cdot,\cdot)$ is not closed unless the manifold is flat. Nevertheless we can ask if there is a different metric on the tangent bundle turning it into a hyperkähler manifold.
Actually the answer to this question is yes (for real-analytic Kähler manifolds), at least in a neighborhood of the zero section as shown by B. Feix \cite{Fx01} and D. Kaledin \cite{Kl01} independently. 

\begin{Theorem}(\cite{Fx01} Thm.\ A)
Let $(M,j,g)$ be a real-analytic Kähler manifold. Then there exists a hyperkähler metric in a neighbourhood of the zero section of the cotangent bundle which is compatible with the canonical holomorphic-symplectic structure.\\
Furthermore, the $S^1$-action $(x,v)\to (x,e^{jt}v)$ rotating the fibers is isometric and the restriction of the hyperkähler metric to the zero section induces the original Kähler metric.
\end{Theorem}

\noindent
We will discuss this theorem for Hermitian symmetric spaces and state a result by Biquard--Gauduchon \cite{BG21} (see Thm.\ \ref{hyperkähler HSS}) that gives explicit formulas for the hyperkähler structure. 

\subsection{Magnetic systems}
In this section we will introduce magnetic systems. All definitions and more details can for example be found in \cite{Bd14}. Let $(M,g)$ be a Riemannian manifold. Additionally pick a closed two form $\sigma\in\Omega^2(M)$. We will refer to the triple $(M,g,\sigma)$ as \emph{magnetic system}. Since $\sigma$ is closed, we can use it to define a twisted symplectic structure on the tangent bundle
$$
\omega_\sigma:=\dd\lambda-\pi^*\sigma.
$$
This indeed defines a symplectic form, as $\omega_\sigma$ is closed,
$$
\dd \omega_\sigma=\dd ^2\lambda-\dd\pi^*\sigma=-\pi^*\dd\sigma=0,
$$
and non-degenerate because it takes the form
$$
\omega_\sigma=\left(\begin{array}{rr} \sigma & -g \\
      g & 0 
      \end{array}\right)
$$
with respect to the splitting $TTM=\mathcal{H}\oplus\mathcal{V}$, where we used the Levi-Civita connection to define the horizontal distribution.
\begin{Definition}[Lorentz force]
The Lorentz force is the bundle map $F:TM\to TM$ defined via 
$$
g_x(F_x(v),w)=\sigma_x(v,w).
$$
\end{Definition}\label{Lorentz force}
\noindent
The Lorentz force determines the Hamiltonian vector field for the kinetic energy with respect to the twisted symplectic form. 
\begin{Lemma}(\cite{BR19} Lemma 6.1)\label{twisted hamiltonian vector field}
Let $E(x,v):=\frac{1}{2}g_x(v,v)$ denote the kinetic energy. The Hamiltonian vector field $X_E$ is given by
$$
(X_E)_{(x,v)}=v^\mathcal{H}+F(v)^\mathcal{V}.
$$
If $\sigma=0$, the Lorentz force vanishes and $X_E=v^\mathcal{H}$, i.e.\ $X_E$ generates the geodesic flow. For $\sigma\neq 0$ the flow is referred to as magnetic geodesic flow.
\end{Lemma}

\noindent 
In the special case where $\sigma$ is exact (i.e.\ $\sigma=\dd\theta$ for some 1-form $\theta\in\Omega^1(M)$) we can shift the zero-section to see that the twisted symplectic form is equivalent to the standard symplectic form. Denote $A\in \mathfrak{X}(M)$ the vector field dual to $\theta$ defined via $g(A,\cdot)=\theta(\cdot)$. Consider the map 
$$
L_A: TM\to TM;\ (x,v)\mapsto (x,v+A_x).
$$
Then 
$$
L_A^*(\lambda)_{(x,v)}=\lambda_{(x,v+A_x)}(\dd L_A \cdot)=g_x(v+A_x,\dd\pi\dd L_A \cdot)=g_x(v,\dd\pi\cdot)+g_x(A_x,\dd\pi\cdot)=\lambda_{(x,v)}+\pi^*\theta,
$$
where we used $\dd\pi\dd L_A=\dd\pi$ as $L_A$ is a bundle map. We see that this map maps the twisted symplectic structure to the standard one. Further the kinetic Hamiltonian transforms as
$$
E(L_A(x,v))=\frac{1}{2}\vert v+A\vert^2\equiv\frac{1}{2}\vert v\vert ^2+ A\cdot v+ V(x),
$$
which has the form of an electro-magnetic Hamiltonian for a charge moving in a magnetic field $B\equiv \mathrm{rot} A$. We call it therefore a magnetic system!\\
\ \\
\noindent
\textbf{Ma\~{n}\'{e}'s critical value.}
In the non-twisted case the geodesic flow on different energy hyper surfaces is conjugated, this is not the case anymore for magnetic systems. Indeed there are values where the dynamics on the energy hyper surface changes dramatically. We shall in examples see that this usually happens at the so called Ma\~{n}\'{e}'s critical value. Our main source for this section is \cite{CFP10}.
\begin{Definition}\label{Mane value via forms}
Consider $(TM,\omega_\sigma)$ for some closed Riemannian manifold $(M,g)$ and a closed two form $\sigma\in\Omega^2(M)$. Denote $\hat M$ the universal cover $M$. Define
$$
c(M,g,\sigma):=\inf_\theta \sup_{x\in\hat M}\hat E(x,{}^g\theta_x),
$$
where $\hat E$ is the lift of $E$, the infimum is taken over primitives of $\hat \sigma$ and ${}^g\theta$ denotes the metric dual of $\theta$. If $\hat \sigma$ is not exact, then $c(M,g,\sigma):=\infty$ by convention.
\end{Definition}
\begin{Remark}
\noindent
One can use different coverings and different Hamiltonians to define other Ma\~{n}\'{e} critical values, but we will restrict to the the universal cover and the kinetic Hamiltonian as in the definition above.
\end{Remark}
\noindent
The Ma\~{n}\'{e}'s critical value can also be defined in terms of the Lagrangian $\hat L$ the Legendre dual of $\hat E$,
$$
\hat L(x,v)=\frac{1}{2}\vert v\vert^2-\theta_x(v).
$$
On an absolute continuous curve $\gamma: [0,T]\to \hat M$ define the action of $\hat L$ as
$$
A_{\hat L}(\gamma)=\int_0^T\hat{L}(\gamma(t),\dot\gamma(t))\dd t.
$$
\begin{Proposition}[\cite{CIPP98}]\label{Mane via action}
The Ma\~{n}\'{e} critical value satisfies
$$
c(M,g,\sigma)=\inf\lbrace k\in \R : A_{\hat L+k}(\gamma) \geq 0\ \text{for any absolutely continuous closed curve}\ \gamma\rbrace.
$$
\end{Proposition}

\subsection{Coadjoint orbits}\label{sec1.1}
Coadjoint orbits are a pleasant class of examples for homogeneous symplectic manifolds. We will later see that also Hermitian symmetric spaces can be realized as coadjoint orbits, but for now let $G$ be a finite dimensional, real, semisimple Lie group and denote by $\mathfrak{g}$ its Lie algebra. All proofs and details of this section can be found in Kirillov's book \cite[Ch. 1]{Kr04}.\\
\ \\
\textbf{Adjoint representation.} We denote by $C_g$ the conjugation by $g\in G$, i.e.\
$
C_g: G\to G;\ \ h\mapsto ghg^{-1}.
$
Then the \emph{adjoint representation} of $G$ on $\mathfrak{g}$ is given by
$
\mathrm{Ad}: G\to \mathrm{GL}(\mathfrak{g}), \ \ g\mapsto \mathrm{Ad}_g:=(\dd C_g)_e.
$
We further denote the induced adjoint representation of the Lie algebra by
$
\mathrm{ad}: \mathfrak{g}\to\mathrm{End}(\mathfrak{g});\ \ X\mapsto \mathrm{ad}_X:=(\dd \mathrm{Ad})_e(X).
$
In fact, 
$
\mathrm{ad}_X(Y)=[X,Y]\in\mathfrak{g}.
$
\\
\ \\
\noindent
\textbf{Killing form.} The \emph{Killing form} is defined as the symmetric bilinear form
\begin{equation*}\label{Killing form}
    B:\mathfrak{g}\times\mathfrak{g}\to \R;\ \ (X,Y)\mapsto \mathrm{tr}(\mathrm{ad}_X\circ\mathrm{ad}_Y).
\end{equation*}
It is invariant under Lie algebra automorphisms, in particular under the adjoint action, and thus $\mathrm{ad_X}$ is skew-symmetric with respect to $B$ for any $X\in \mathfrak{g}$. Moreover, $B$ is in our cases non-degenerate as $\mathfrak{g}$ is semisimple. \\
\ \\
\textbf{Coadjoint representation.}
The \emph{coadjoint representation} of $G$ on the dual Lie algebra $\mathfrak{g}^*$ is denoted by
$$
\mathrm{Ad}^*: G\to \mathrm{Gl}(\mathfrak{g}^*);\ g\mapsto \mathrm{Ad}^*_g.
$$
It is the dual of the adjoint representation and thus implicitly defined via 
$$
\langle \mathrm{Ad}^*_g F,X\rangle=\langle  F,\mathrm{Ad}_{g^{-1}}X\rangle,\ \ \forall F\in \mathfrak{g}^*\ \forall X\in\mathfrak{g},
$$
where $\langle\cdot,\cdot\rangle: \mathfrak{g}^*\times\mathfrak{g}\to\R$ denotes the natural pairing. We use non-degeneracy of the Killing form to identify $\mathfrak{g}^*\cong\mathfrak{g}$. It intertwines adjoint and coadjoint action as 
\begin{align*}
\langle \mathrm{Ad}^*_g F,Y\rangle&:=\langle  F,\mathrm{Ad}_{g^{-1}}Y\rangle=B(X_F,\mathrm{Ad}_{g^{-1}}Y)=B(\mathrm{Ad}_{g}X_F, Y),
\end{align*}
for any $Y\in\mathfrak{g}$ and $F\in\mathfrak{g}^*$ with dual $X_F\in\mathfrak{g}$.
Hence, we will from now on use adjoint and coadjoint descriptions interchangeably.\\
\ \\
\textbf{Coadjoint orbits.}
The orbit $O_p:=\mathrm{Ad}^*_G(p)\subset \mathfrak{g}^*$ of a point $p\in\mathfrak{g}^*$ under the coadjoint action is called \emph{coadjoint orbit}. It can be identified with the homogeneous space $G/G_p$ ,where $G_p$ is the stabilizer of $p$, via the isomorphism
$$
\mathrm{Ad}^*_g(p)\mapsto g\cdot G_p.
$$
For an element of the Lie-algebra $a\in\mathfrak{g}$, we define the induced vector field at $x\in O_p$ as 
$$
(a)^\#_x:=\diff \mathrm{Ad}^*_{\exp ta}x.
$$
As the coadjoint action is transitive on $O_p$, we can represent every vector in $T_xO_p$ in this way by an element of $\mathfrak{g}$. Using the identification of adjoint and coadjoint orbit via the Killing form, we obtain the following identifications
$$
(a)^\#_x=\diff \mathrm{Ad}_{\exp ta}x=[a,x]\ \ \text{and} \ \ T_x O_p\cong [\mathfrak{g}, x].
$$
\ \\
\textbf{Kirillov--Kostant--Souriau form.}
Coadjoint orbits carry a natural $G$-invariant symplectic structure, called \emph{KKS-form}\footnote{KKS stands for Kirillov--Kostant--Souriau.}. At $p\in O_p$ it is for $a,b\in\mathfrak{g}$ given by
$$
\sigma_p(a^\#_p,b^\#_p):=-\langle p, [a,b]\rangle,
$$
where the natural pairing $\langle\cdot,\cdot\rangle:\mathfrak{g}^*\times\mathfrak{g}\to \R$ is extended equivariantly to a symplectic form on $O_p$. Well-definedness of this definition and non-degeneracy follow from the fact that the kernel of $\langle p,[\cdot,\cdot]\rangle:\mathfrak{g}\times\mathfrak{g}\to\R$ is precisely $\mathfrak{g}_p$, the Lie-algebra of $G_p$, and invariance under $G_p$. Closedness follows from the Jacobi identity. Using the Killing form one can push the symplectic structure to the adjoint orbit, also denoted by $\sigma$, e.g.
$$
\sigma_x(a^\#_x,b^\#_x):=-B (x, [a,b])\ \ \forall x\in O_p\ \ \forall a,b\in \mathfrak{g}.
$$
\ \\
\textbf{Momentum map.}
For a symplectic group action 
$
\Psi:G\to \mathrm{Symp}(N,\sigma);\ g\mapsto \Psi_g
$
on a general symplectic manifold $(N,\omega)$, a map
$
\mu: N\to \mathfrak{g}^*
$
is called \emph{momentum map} if 
$$
\dd \langle \mu, a\rangle=\iota_{a^\#}\omega \ \ \ \forall a\in \mathfrak{g}.
$$
We additionally require $\mu$ to be equivariant with respect to the coadjoint action, i.e.
$$
\mu(\Psi_g(x))=\mathrm{Ad}^*_g(\mu(x))\ \ \ \forall x\in O_p,\ \forall g\in G.
$$
A symplectic action that admits a momentum map is called \emph{Hamiltonian}.\\
\ \\
Indeed the obvious symplectic action of $G$ on $O_p$ is Hamiltonian and the momentum map is given by the inclusion
$$
\mu: O_p\hookrightarrow \mathfrak{g}\cong \mathfrak{g}^*.
$$
\begin{proof}
Since the pairing is $\mathrm{Ad_G}$ invariant we have
\begin{align*}
    0&=\diff B(\mathrm{Ad}_{\exp t\xi_1} x,\mathrm{Ad}_{\exp t\xi_1} \xi_2)=B([\xi_1,x],\xi_2)+B(x,[\xi_1,\xi_2])\\
    &=\diff B(\mu(\mathrm{Ad}_{\exp t\xi_1}x), \xi_2)-\sigma_x([x,\xi_1],[x,\xi_2]),
\end{align*}
for all $x\in O_p$ and $\xi_1,\xi_2\in\mathfrak{g}$.
From here we immediately see
$$
\dd \langle \mu, \xi\rangle = \iota_{\xi^\#}\sigma\ \ \forall\xi\in\mathfrak{g},
$$
since the (co-)adjoint action is transitive on the orbit.
\end{proof}

\subsection{Riemannian symmetric spaces}
In this section, we collect some basic material of Riemannian symmetric spaces. The sources are mainly \cite{HG01} and \cite{Whd05}. 

\begin{Definition}\label{symmetric space}
A Riemannian symmetric space is a connected Riemannian manifold $(M,g)$ with the property that the geodesic reflection at any point is an isometry of $M$. Explicitly this means for any point $p\in M$ there is an isometry $s_p: M\to M$ that satisfies
$$
s_p(p)=p\ \ \text{and} \ \ (\dd s_p)_p=-\mathrm{id}.
$$
\end{Definition}
\noindent
Symmetric spaces are complete as we can use the geodesic symmetry to extend any geodesic segment to infinite length.
Because $M$ is connected, Hopf-Rinow's theorem implies that any two points $p$ and $q$ in $M$ can be joined by a geodesic $\gamma:\R\to M$ satisfying $\gamma(0)=p$ and $\gamma(t)=q$ for some $t\in\R$. It follows that $s_{\gamma(t/2)}(p)=q$ thus the connected component of the identity of the isometry group $G=\mathrm{Is}^0(M)$ acts transitively. We can realize $M$ as homogeneous space $M=G/K$ by picking a base point $o\in M$ and calling the stabilizer $\mathrm{Stab}_G(o)=:K$.\\
\ \\
\textbf{Orthogonal symmetric Lie algebras \cite[Ch.\ V.1]{HG01}.}
The involutive isometry $s_o$ induces an involution on $G$ via $g\mapsto s_{o}\circ g\circ s_{o}$ and its differential is a so called \emph{Cartan involution} $\theta:\mathfrak{g}\to\mathfrak{g}$. As an involution, it has eigenvalues $\pm 1$ and we split $\mathfrak{g}$ into the eigenspaces of $\theta$, i.e.\ $\mathfrak{g}=\mathfrak{k}\oplus\mathfrak{p}$ such that $\theta\vert_\mathfrak{k}=1,\ \theta\vert_\mathfrak{p}=-1$. Moreover, $\theta$ is a Lie algebra automorphism leaving the Killing-form $B$ invariant. The decomposition is thus orthogonal with respect to $B$. Further, being a Lie algebra automorphism implies the commutator relations 
\begin{equation}\label{commutator p,k}
[\mathfrak{k},\mathfrak{k}]\subset \mathfrak{k}, \ \ [\mathfrak{k},\mathfrak{p}]\subset \mathfrak{p}\ \ \text{and}\ \ [\mathfrak{p},\mathfrak{p}]\subset \mathfrak{k}.
\end{equation}
 The Lie subalgebra $\mathfrak{k}$ can be identified with the Lie algebra of $K$, thus is compact. 
 Further the differential of $\pi:G\to M$ has kernel $\mathfrak{k}$ and thus induces an identification $$\dd\pi_e\vert_\mathfrak{p}:\mathfrak{p}\xrightarrow{\sim} T_oM.$$
 The pair $(\mathfrak{g},\theta)$ is what is called an \emph{orthogonal symmetric Lie algebra} (short OSLA).
 \begin{Definition}
 An orthogonal symmetric Lie algebra (OSLA) is a pair $(\mathfrak{g},\theta)$ where
 \begin{itemize}
     \item $\mathfrak{g}$ is a real Lie algebra,
     \item $\theta\in\mathrm{End}(\mathfrak{g})$ such that $\theta^2=\mathrm{id}$ but $\theta\neq\mathrm{id}$,
     \item $\mathfrak{k}:=E_1(\theta)$\footnote{The eigenspace of $\theta$ for the eigenvalue 1.} is compact \footnote{A Lie algebra is called compact if its Killing form is negative definite.}.
 \end{itemize}
 \end{Definition}
\ \\
\textbf{Irreducibility \cite[Ch.\ VIII.5]{HG01}.}
A symmetric space $M$ is called \emph{irreducible} if it does not split as a product $M=M_1\times M_2$ of symmetric spaces $M_1, M_2$. Irreducibility can also be seen algebraically. 
\begin{Definition}\label{irreducible OSLA}
An OSLA $(\mathfrak{g},\theta)$ is irreducible if
\begin{itemize}
    \item $\mathfrak{g}$ is semisimple,
    \item $\mathfrak{k}$ contains no ideal of $\mathfrak{g}$,
    \item the Lie algebra $\mathrm{ad}_\mathfrak{g}(\mathfrak{k})$ acts irreducibly on $\mathfrak{p}$.
\end{itemize}
\end{Definition}
\noindent
Indeed, if $M$ is irreducible then so is the associated OSLA.\\
\ \\
\textbf{Duality -- compact vs. non-compact type \cite[Ch.\ V.2]{HG01}.}
Irreducible Hermitian symmetric spaces fall into two types (\emph{compact} and \emph{non-compact}) according to their OSLA's. 
\begin{Definition}\label{type}
An irreducible OSLA is called of compact type if the Killing form is negative definite. It is called of non-compact type if the Killing form restricted to $\mathfrak{p}\times\mathfrak{p}$ is positive definite.
\end{Definition}
\noindent
Indeed, one can show that every irreducible OSLA is either of compact or non-compact type. 
These types are dual in the following sense. We consider the complexification $\mathfrak{g}^\C$ of $\mathfrak{g}$ and the natural complexification $\theta^\C$ of $\theta$ and define 
$$
(\mathfrak{g}^\vee,\theta^\vee):=(\mathfrak{k}\oplus i\mathfrak{p},\theta^\C\vert_{\mathfrak{k}+i\mathfrak{p}}).
$$
If $(\mathfrak{g},\theta)$ is an OSLA of compact type, then $(\mathfrak{g}^\vee,\theta^\vee)$ is an OSLA of non-compact type and the other way around \cite[Prop. 2.1, Ch. V]{HG01}.\\
\ \\
\textbf{Euclidean type.}
In principle there is a third type of symmetric spaces, called \emph{Euclidean type}. This type occurs if $\mathfrak{p}$ is an abelian ideal of $\mathfrak{g}$. Indeed all symmetric spaces of Euclidean type can be isometrically identified with an Euclidean space. Furthermore, every symmetric space $M$ can be decomposed as a product \cite[Prop. 4.2, Ch. V]{HG01} 
$$
M=M_0\times M_-\times M_+,
$$
where $M_0$ is a Euclidean space and $M_+$, $M_-$ are
symmetric spaces of the compact and non-compact type, respectively.\\
\ \\
\textbf{Uniqueness of invariant metric.}
Irreducibility also ensures that there is (up to scalar multiple) a unique $G$-invariant metric, i.e. any invariant metric is induced by $\pm\lambda^2 B\vert_{\mathfrak{p}\times\mathfrak{p}}$ for some real constant $\lambda\neq 0$ and the sign chosen such that $\pm B\vert_{\mathfrak{p}\times\mathfrak{p}}$ is positive definite. In particular the $G$-invariant metric $g$ we started with is of this form. To see this assume there was another one. This would define another $K$-invariant scalar product $\langle\cdot,\cdot\rangle$ on $\mathfrak{p}$. Now we define a symmetric operator $S:\mathfrak{p}\to\mathfrak{p}$ implicitly via
$$
\langle\cdot,\cdot\rangle=B\vert_\mathfrak{p}(S\cdot,\cdot).
$$
As both scalar products are $K$-invariant $S$ must commute with all elements of $K$. Now $S$ is symmetric and therefore diagonalizable over $\R$. All eigenspaces are $K$-invariant subspaces, but the action of $K$ on $\mathfrak{p}$ is irreducible so $S$ must be of the form $\lambda^2\cdot\mathrm{id}$ for some $\lambda\neq 0$.\\
\ \\
\textbf{Curvature of symmetric spaces \cite[Ch.\ IV.4]{HG01}.}
As the sectional curvature is invariant under isometries, it is enough to determine sectional curvature at the base point. At the base point $o$ we can exploit the fact that $T_oM\cong \mathfrak{p}$. We can thus express the curvature tensor $R$ of $M$ at $o$ in terms of the curvature tensor of $G$ (\cite[Thm. 4.2, Ch. IV]{HG01}), i.e.\
\begin{equation}\label{curvature}
R(a,b)c (o)=-[[a,b],c]\ \ \ \forall a,b,c\in\mathfrak{p}\cong T_oM.
\end{equation}
The Riemannian metric $g$ on $M$ is induced by $+\lambda^2 B$ or $-\lambda^2 B$ where $M$ is of compact, respectively non-compact type and some real constant $\lambda\neq 0$. As
$$
g_o(a, R(a,b)b)=\mp \lambda^2 B(a,[[a,b],b])=\pm \lambda^2 B([a,b],[a,b]) \ \ \ \forall a,b\in\mathfrak{p}
$$
and $B\vert_{\mathfrak{t}\times\mathfrak{t}}>0$, it follows that if $M$ is of compact resp.\ non-compact type it has non-negative resp.\ non-positive sectional curvature. The sectional curvature of spaces of Euclidean type vanishes identically \cite[Thm.\ 3.1, Ch.\ V.3]{HG01}. \\
\ \\
\textbf{Totally geodesic subspaces \cite[Ch.\ IV.7]{HG01}.}
A submanifold $N\subset M$ is \emph{geodesic} at $p\in N$ if for all $v\in T_p N\subset T_p M$ the $M$-geodesic with tangent $v$ is contained in $N$. The submanifold is \emph{totally geodesic} if it is geodesic at every point in $N$. A totally geodesic submanifold of a symmetric space is itself a symmetric space as the geodesic symmetries restrict to the submanifold. The algebraic characterization of totally geodesic submanifolds of symmetric spaces leads to the notion of \emph{Lie triple systems}.
\begin{Definition}\label{Lie triple system}
Let $\mathfrak{g}$ be a Lie algebra. A Lie triple system is a vector space $\mathfrak{n}\subset \mathfrak{g}$ such that
$$
[\mathfrak{n},[\mathfrak{n},\mathfrak{n}]]\subset\mathfrak{n}.
$$
\end{Definition}
\noindent
Indeed, $N\subset M$ is a totally geodesic submanifold containing $o$ if $\mathfrak{n}:=\dd\pi^{-1}(T_o N)\subset\mathfrak{p}$ is a Lie triple system. Conversely, if $\mathfrak{n}\subset \mathfrak{p}$ is a Lie triple system, then $N:=\exp_o\dd\pi (\mathfrak{n})$ is a totally geodesic submanifold \cite[Thm.\ 7.2, Ch.\ IV]{HG01}.\\
\ \\
\textbf{Maximal flat subspaces \cite[Ch.\ V.6]{HG01}.}
A Riemannian manifold $F$ is \emph{flat} if all sectional curvatures vanish identically. A \emph{maximal flat} $F\subset M$ is a flat submanifold that is not contained in a flat submanifold of higher dimension. Via the exponential map, maximal flats $F$ correspond one-to-one to maximal abelian subalgebras $\mathfrak{a}\subset \mathfrak{p}$ \cite[Thm. 6.1, Ch. V]{HG01}.
\begin{theorem}[\cite{HG01} Thm. 6.2, Ch. V]\label{transitive on maximal flats}
Let $M$ be a symmetric space of compact/ non-compact type and $\mathfrak{a}_1,\mathfrak{a}_2\subset \mathfrak{p}$ maximal abelian subalgebras, then there exists a $k\in K$ such that 
$$
\mathrm{Ad}_k(\mathfrak{a}_1)=\mathfrak{a}_2.
$$
\end{theorem}
\noindent
The Theorem yields a well-defined notion of the rank of a symmetric space.
\begin{Definition}\label{rank}
The rank of $M$ is the dimension of maximal flats.
\end{Definition}
\noindent
\textbf{Locally symmetric spaces.}
We call a Riemannian manifold $(M,g)$ \emph{locally symmetric} if it is isometrically covered by a symmetric space. Locally symmetric spaces can also be characterized by the following theorem.
\begin{Theorem}[\cite{HG01} Thm.\ 1.1, Ch.\ IV]\label{locally symmetric}
    A Riemann manifold $(M,g)$ is locally symmetric if and only if the Riemannian curvature tensor is parallel, i.e. $\nabla R=0$ where $\nabla$ denotes the Levi-Civita connection.
\end{Theorem}

\subsection{Hermitian symmetric spaces}\label{HSS}
Until now we did not discuss any relations between symmetric spaces and symplectic manifolds. We will see that Hermitian symmetric spaces are precisely at the intersection, i.e.\ they are symmetric spaces with an $G$-invariant symplectic form. 

\begin{Definition}
A Hermitian symmetric space is a connected complex manifold with Hermitian structure $(M,g,j)$, such that the geodesic reflection at any point is a holomorphic isometry of $M$. Explicitly this means for any point $p \in M$ there is a holomorphic isometry $s_p : M \to M$ that satisfies
$$
s_p (p) = p\ \  \& \ \ (\dd s_p)_p = -\mathrm{id}.
$$
\end{Definition}

\noindent
All Hermitian symmetric spaces are symmetric spaces, so everything discussed in the previous section continues to hold. Nevertheless the class of Hermitian symmetric spaces is much smaller than the class of symmetric space. Indeed irreducible Hermitian symmetric spaces can be characterized by the following theorem. 
\begin{theorem}[\cite{HG01} Thm. 6.1. Ch. VIII]\label{characterization HSS}
\begin{itemize}
    \item[(i)] The compact irreducible Hermitian symmetric spaces are exactly the manifolds $G/K$ where $G$ is a connected compact simple Lie group with center $\lbrace e\rbrace$ and $K$ has nondiscrete center and is a maximal connected proper subgroup of $G$.
    \item[(ii)] The noncompact irreducible Hermitian symmetric spaces are exactly the manifolds $G^\vee/K$ where $G^\vee$ is a connected noncompact simple Lie group with center $\lbrace e\rbrace$ and $K$ has nondiscrete center and is a maximal compact subgroup of $G^\vee$.
\end{itemize}
\end{theorem}
\noindent
The center of $K$ can be described more accurately in both cases compact and noncompact. 
\begin{Proposition}[\cite{HG01} Thm. 6.1. Ch. VIII]\label{center is circle}
The center $C(K)$ of the group $K$ in Theorem \ref{characterization HSS} (i) and (ii) is analytically isomorphic to the circle group.
\end{Proposition}
\noindent
\textbf{Uniqueness of invariant complex structure.}
By Prop.\ \ref{center is circle} we can identify $C(K)\cong S^1$ and so the Lie algebra of $C(K)$ is identified with $i\R$. Denote $Z\in \mathfrak{g}$ the element that corresponds to $i$ under this identification. Now $A:=\mathrm{ad}_Z$ is an antisymmetric endomorphism of $\mathfrak{p}$ which is also $\mathrm{Ad}_K$-invariant. Thus $A^2$ is symmetric, negative definite and $\mathrm{Ad}_K$-invariant. By the same argument (similar to Schur's lemma) as in the proof of uniqueness of the invariant metric, this means $\mathrm{ad}_Z^2=-\lambda^2\mathrm{id}_\mathfrak{p}$ for some $\lambda\in\R$. Finally, $\lambda=1$ because $e^{2\pi t A}$ has eigenvalue $e^{\pm 2\pi t i\lambda}$ and for $t=1$ this eigenvalue must be equal to $1$ since $e^{2\pi A}$ is the identity.\\
\ \\
Thus $j_o=\mathrm{ad}_Z$ defines a complex structure on $T_oM\cong \mathfrak{p}$. Observe that $j_o$ is $K$-invariant as for all $k\in\mathfrak{k}$ and $v\in T_o M\cong\mathfrak{p}$ we have
$$
\mathrm{ad}_k(j_o (v))=\mathrm{ad}_k([Z,v])=-[Z,[v,k]]-[v,[k,Z]]= [Z, \mathrm{ad}_k v]=j_o(\mathrm{ad}_k(v)).
$$
Therefore we can extend $j_o$ equivariantly to a $G$-invariant almost complex structure $j$ on $M$.
Now the invariant metric $g$ and the invariant almost complex structure $j$ determine an invariant Hermitian metric $h$ on $TM^\C$. Analogously to the proof of uniqueness of the Riemannian metric $g$ one shows that $G$-invariant Hermitian metric $h$ on $TM^\C$ is also unique up to scalar multiple. To do so observe that $j$ promotes the adjoint representation to a complex irreducible representation on $\mathfrak{p}^\C$. In particular one can directly apply Schur's lemma to obtain uniqueness. As $g$ and $h$ are unique up to scalar multiple, the complex structure $j$ is unique up to sign. In particular the almost complex structure $j$ must up to sign coincide with the complex structure we started with and is therefore integrable.  \\
\ \\
\noindent
\textbf{Intermezzo -- Root systems.}
Let $\mathfrak{g}^\C$ be a semi-simple complex Lie algebra. A \emph{Cartan subalgebra} $\mathfrak{h}^\C$ is a maximal abelian subalgebra such that, for each $h\in \mathfrak{h}^\C$, $\mathrm{ad}_h$ is diagonalizable.  In particular, the operators $\mathrm{ad}_h$ can be diagonalized simultaneously. This leads to the definition of roots and root spaces.
\begin{Definition}\label{root system}
A root of $(\mathfrak{g}^\C,\mathfrak{h}^\C)$ is a non-zero linear form $\alpha:\mathfrak{h}^\C\to\C$ such that the corresponding root space 
$$
\mathfrak{g}_\alpha:=\lbrace X\in\mathfrak{g}^\C\ \vert\  \mathrm{ad}_hX=\alpha(h) X\ \forall\ h\in \mathfrak{h}^\C\rbrace
$$
is nonzero. Denote the set of roots by $\Delta$.
\end{Definition}
\noindent
The root spaces are simultaneous eigenspaces of $\mathrm{ad}_h$, for all $h\in\mathfrak{h}^\C$, and we get a decomposition of $\mathfrak{g}^\C$ into the direct sum of root spaces
$$
\mathfrak{g}^\C=\mathfrak{g}_0\oplus \bigoplus_{\alpha\in\Delta}\mathfrak{g}_\alpha
$$
The subspace $\mathfrak{g}_0$ is the centralizer of $\mathfrak{h}^\C\subset\mathfrak{g}^\C$. As for example explained in \cite[Ch. 3, Thm.\ 3]{S00}, we have that
$
\mathfrak{g}_0=\mathfrak{h}^\C.
$
Note that using the Jacobi-identity $[\mathfrak{g}_\alpha,\mathfrak{g}_\beta]\subset\mathfrak{g}_{\alpha+\beta}$, in particular $\mathfrak{h}_\alpha:=[\mathfrak{g}_\alpha,\mathfrak{g}_{-\alpha}]\subset \mathfrak{h}^\C$ for all $\alpha,\beta\in\Delta$. Indeed, one can show (\cite[Ch. VI, Thm. 2]{S00}) that for all roots $\alpha\in\Delta$ the spaces $\mathfrak{g}_\alpha$ and $\mathfrak{h}_\alpha$ are one-dimensional. Then there exists a unique element $H_\alpha\in\mathfrak{h}_\alpha$ determined by $\alpha(H_\alpha)=2$. It is easy to see that for each non-zero element $X_\alpha\in \mathfrak{g}_\alpha$ there exists an element $Y_\alpha$ in $\mathfrak{g}_{-\alpha}$ such that
$$
[H_\alpha, X_\alpha]= 2 X_\alpha,\ \ \ [H_\alpha, Y_\alpha]=-2 Y_\alpha\ \ \text{and}\ \ [X_\alpha, Y_\alpha]=H_\alpha.
$$
These elements generate a copy of $\mathfrak{sl}(2,\C)$ that we shall denote by $\mathfrak{g}[\alpha]$.\\
\ \\
\textbf{Polyspheres/ Polydiscs.}
We go back to our set up before the intermezzo. Denote by $\mathfrak{g}^\C$ the complexification of $\mathfrak{g}$. It decomposes as
$$
\mathfrak{g}^\C=\mathfrak{k}^\C\oplus\mathfrak{p}_+\oplus\mathfrak{p}_-,
$$
where $\mathfrak{k}^\C$ is the complexification of $\mathfrak{k}$ and $\mathfrak{p}_\pm$ are the $\pm i$-eigenspaces of the complex linear extension of $j=\mathrm{ad}_Z$. Assume without loss of generality that $\mathfrak{g}=\mathfrak{k}\oplus\mathfrak{p}$ is the compact real form and denote by $\mathfrak{g}^\vee=\mathfrak{k}\oplus i\mathfrak{p}$ its non-compact dual. As described in \cite[Ch. VIII.7]{HG01} the maximal abelian subalgebra $\mathfrak{h}\subset \mathfrak{k}$ complexifies to a Cartan subalgebra $\mathfrak{h}^\C$ of $\mathfrak{g}^\C$. We denote by $\Delta$ the set of roots of $\mathfrak{g}^\C$ with respect to $\mathfrak{h}^\C$. Since $[\mathfrak{h}^\C,\mathfrak{k}^\C]\subset \mathfrak{k}^\C$ and $[\mathfrak{h}^\C,\mathfrak{p}^\C]\subset \mathfrak{p}^\C$ the root space $\mathfrak{g}_\alpha$ is either contained in $\mathfrak{k}^\C$ or $\mathfrak{p}^\C$. The roots are called \emph{compact} or \emph{non-compact}, respectively. In particular,
$$
\mathfrak{k}^\C=\mathfrak{h}^\C\oplus\bigoplus_\alpha \mathfrak{g}_\alpha,\ \ \ \ \mathfrak{p}^\C=\bigoplus_\beta \mathfrak{g}_\beta,
$$
where $\alpha$ runs over compact roots and $\beta$ runs over all non-compact roots. Moreover, one can partition into positive and negative non-compact roots according to the sign of $-i\beta(Z)$. Indeed, this is compatible with the decomposition $\mathfrak{p}^\C=\mathfrak{p}_+\oplus\mathfrak{p}_-$ and we can write
$$
\mathfrak{p}_+=\bigoplus_{\beta} \mathfrak{g}_\beta,\ \ \ \ \mathfrak{p}_-=\bigoplus_{\beta} \mathfrak{g}_{-\beta},
$$
where $\beta$ runs over all positive non-compact roots. Two roots $\alpha, \beta\in \Delta$ are called \emph{strongly orthogonal} if $\alpha\pm\beta\notin \Delta$, which implies $[\mathfrak{g}_\alpha, \mathfrak{g}_\beta]=0$. By \cite[Prop. 7.4, Ch. VIII]{HG01} there exist strongly orthogonal positive non-compact roots $\gamma_1,\ldots,\gamma_r$. Thus, the subspace
$$
\bigoplus_{i=1}^r\mathfrak{g}^\C[{\gamma_i}]\subset \mathfrak{g}^\C
$$
is isomorphic to $\mathfrak{sl}(2,\C)^r$. The intersection with $\mathfrak{g}$ resp.\ $\mathfrak{g}^\vee$ yield subalgebras of $\mathfrak{g}$ isomorphic to the compact real form $\mathfrak{su}(2)^r$ resp.\ the dual non-compact real form $\mathfrak{sl}(2,\R)^r$. Intersecting these with $\mathfrak{p}$ resp.\ $i\mathfrak{p}$ yield Lie-triple systems in $\mathfrak{g}$ resp.\ $\mathfrak{g}^\vee$ and thus realize totally geodesically embedded polyspheres resp.\ polydiscs of $G/K$ resp.\ $G^\vee/K$. Indeed Hermitian symmetric space corresponding to $\mathfrak{su}(2)$ is $(\cp^1)^r$, while the Hermitian symmetric space corresponding to $\mathfrak{sl}(2,\R)$ is $(\C\mathrm{H}^1)^r$. These polysphere resp.\ polydisc obtained be integrating the copy $\mathfrak{su}(2)^r$ resp.\ $\mathfrak{sl}(2,\R)^r$ can be translated by the adjoint action of $G$ to see that there is a polysphere resp.\ polydisc through every point. In total we obtain the polysphere resp.\ polydisc theorem.

\begin{theorem}[Polysphere/ polydisc theorem \cite{W72}, p.\ 280]\label{polysphere theorem}
Let $M$ be an irreducible Hermitian symmetric space of rank $r$. For any point $q=(x, v) \in TM$, there exists a point $p=(x_0, v_0)\in T\Sigma^r$ and a holomorphic totally geodesic embedding  
$$
\iota_{p,q}:\Sigma^r=\Sigma\times\ldots\times\Sigma\hookrightarrow M
$$
such that 
$$
\iota_{p,q}(x_0)=x\ \  \ \text{and}\ \ \ (\dd\iota_{p,q})_{x_0}v_0=v.
$$
Here, $\Sigma=\cp^1$ in the compact case and $\Sigma=\C \mathrm{H}^1$ in the non-compact case. 
\end{theorem}

\begin{Remark}\label{equivariance polysphere}
    These embeddings are equivariant in a double sense. We denote by $H$ either $\mathrm{SU}(2)$ in the compact case or $\mathrm{SL}(2,\R)$ in the non-compact case.\\
    \ \\
    \textbf{Translation and reparametrization:}
    For all $g\in G$ and $h\in H^r$ the following diagram commutes
    \[ \begin{tikzcd}
\Sigma^r \arrow[hookrightarrow]{r}{ \iota_{p,q}} \arrow[rightarrow]{d}{h} & M \arrow[rightarrow]{d}{g} \\%
\Sigma^r  \arrow[hookrightarrow]{r}{\iota_{hp,gq}}&  M
\end{tikzcd}.
\]
Here the action $h:\Sigma\to \Sigma$ should be interpreted as reparametrization, while the arrow $g: M\to M$ translates a polysphere resp.\ polydisc through $q=(x,v)$ to a polysphere resp.\ polydisc through $gq=(g(x),\dd g(x))$. We say a polysphere resp.\ polydisc goes through $q=(x,v)$ if it goes through $x$ and $v$ is tangent to it.\\
\ \\
\textbf{$H^r$-equivariance of $\iota_{p,q}$:} As discussed above Theorem \ref{polysphere theorem} every embedding $\iota_{p,q}$ comes from a Lie algebra monomorphism $k:\mathfrak{h}^r\hookrightarrow\mathfrak{g}$. This can be integrated to a monomorphism of Lie groups we denote by $R: H^r\hookrightarrow G$. Then $\iota_{p,q}$ is also equivariant with respect to $R$, i.e. for all $h\in H^r$ the following diagram commutes
\[ \begin{tikzcd}
\Sigma^r \arrow[hookrightarrow]{r}{ \iota_{p,q}} \arrow[rightarrow]{d}{h} & M \arrow[rightarrow]{d}{R(h)} \\%
\Sigma^r  \arrow[hookrightarrow]{r}{\iota_{p,q}}&   M
\end{tikzcd}.
\]
\end{Remark}
\begin{Remark}
    From now on we will not need root systems again. One nice thing about the proofs in this paper is, that they only use the polysphere/ polydisc theorem and no root systems explicitly. They were only included to convince the reader that the polysphere/ polydisc theorem holds, but one could equally well just use this theorem as a black box.
\end{Remark}
\ \\
\noindent
\textbf{Foliation of the tangent bundle of a Hermitian symmetric space.}
The polysphere resp.\ polydisc theorem tells us that for every point $(x,v)\in TM$ there is a polysphere resp.\ polydisc through $x$ with $v$ tangent to it. We want to investigate in what sense these $T\Sigma^r$, where $\Sigma=\cp^1$ resp.\ $\Sigma=\C\mathrm{H}^1$, form a foliation of $TM$. It is not hard to see that through some points (for example points on the zero section) go more than one $T\Sigma^r$. So our foliation will be singular, but we can characterize an open dense set of points, where the foliation is not singular.
\begin{Definition}\label{regularpoint}
An element $v\in \mathfrak{p}$ is called \emph{regular} if its centralizer
$$
Z_{\mathfrak{p}}(v):=\lbrace w\in \mathfrak{p}\ \vert \ [v,w]=0\rbrace
$$
has dimension as small as possible. 
\end{Definition}
\noindent
The smallest possible dimension is equal to the rank of M, because $Z_{\mathfrak{p}}$ can be identified with the union of all maximal abelian subalgebras $\mathfrak{a}\subset\mathfrak{p}$ containing $v$, denoted by 
$
\cup\mathfrak{a}.
$
The inclusion $\cup\mathfrak{a}\subset Z_{\mathfrak{p}}(v)$ is immediate. On the other hand any element $w\in Z_{\mathfrak{p}}(v)$ satisfies $[v,w]=0$ and can thus be extended to a maximal abelian subspace $\mathfrak{a}$ containing $v$ and $w$.
In particular regular vectors lie in a unique maximal abelian subspace explicitly given by
$$
\mathfrak{a}_v=Z_{\mathfrak{p}}(v).
$$
One can show that the set of regular vectors is open and dense \cite[Prop. 1, Ch.III.2]{S00}.
We call a point $(x,v)\in TM$ \emph{regular} if $\mathrm{Ad}_g v\in\mathfrak{p}$ is regular for $g\in G$ such that $\mathrm{Ad}_g(x)=o$. The set of regular points is denoted by $T^{\mathrm{reg}}M$.
Observe that picking $r$ vectors, each tangent to a factor in the polysphere resp.\ polydisc through $x$, we obtain a maximal abelian subspace of $\mathrm{Ad}_{g^{-1}}\mathfrak{p}$. Thus there is up to reparametrization only a unique polysphere resp.\ polydisc through a regular point $(x,v)\in T^{\mathrm{reg}}M$.   
\noindent
As all maximal flats are conjugate (see Thm. \ref{transitive on maximal flats}) the same holds true for their complexifications the polyspheres/ polydiscs.
\begin{theorem}\label{transitive on polyspheres}
Every polysphere/ polydisc through $o$ can be mapped to any other polysphere/ polydisc through $o$ by an element of $K$.
\end{theorem}
\noindent
In view of Remark \ref{equivariance polysphere} we obtain a smooth foliation of $T^{\mathrm{reg}}M$.
In particular, we can locally in a neighborhood $U\subset T^{\mathrm{reg}}M$ of a regular point $(x,v)$ define the projections
$$
\pi_i: U\to T\Sigma
$$
on the i-th factor of the product $T\Sigma^r$. In addition the following quantities are locally well-defined and smooth
$$
v_i:=\pi_i(v),\ \ \ r_i:=\vert v_i\vert,\ \ \ X_i:=(v_i)^\mathcal{H}, \ \ \ H_i:=(jv_i)^\mathcal{H},\ \ \ Y_i:=(v_i)^\mathcal{V},\ \ \ V_i:=(jv_i)^\mathcal{V}.
$$
There are two distributions on $T^{\mathrm{reg}}M$ that we will need later, denote
$$
\Upsilon:=\mathrm{span}_{\R}\lbrace Y_1,\ldots,Y_r\rbrace \subset TT^{\mathrm{reg}}M\ \ \ \text{and}\ \ \ \mathcal{D}:=\lbrace a^\#\vert a\in\mathfrak{g}\rbrace\subset TT^{\mathrm{reg}}M.
$$
At $p=(o,v)\in T^{\mathrm{reg}}M$ for any $v\in T_o^{\mathrm{reg}}M$ we can identify the following sub spaces
\begin{equation}\label{Ups und Delta distribution}
\Upsilon_p=\mathrm{span}_\R\lbrace Y_1,\ldots,Y_r\rbrace\vert_{(o,v)} \cong (\mathfrak{a}_v)^\VV
\ \ \ \text{and}\ \ \ 
\mathcal{D}_p=\mathfrak{p}^\HH\oplus[\mathfrak{k},v]^\VV.
\end{equation}
\begin{Lemma}\label{lem8}
If $v\in\mathfrak{p}$ is regular, then
$$
\mathfrak{a}^\perp_v=[\mathfrak{k},v],
$$
where $\perp$ denotes the orthogonal with respect to the Killing form $B$.
\end{Lemma}
\begin{proof}
Take $k\in\mathfrak{k}$, then for all $w\in\mathfrak{a}_v$ we have
$$
B([k,v],w)=-B(k,[w,v])=0
$$
as $[w,v]=0$. Thus $[k,v]\in\mathfrak{a}_v^\perp$. It remains to be shown that $\mathfrak{a}^\perp_v\subset[\mathfrak{k},v]$. We show instead $[\mathfrak{k},v]^\perp\subset\mathfrak{a}_v$. For this take $w\in [\mathfrak{k},v]^\perp$, then
\begin{align*}
    &B([k,v],w)=0\  \ \forall k\in\mathfrak{k}
    \Rightarrow\ \ \ B(k,[v,w])=0\  \ \forall k\in\mathfrak{k}
    \Rightarrow\ \ \ [v,w]\in\mathfrak{k}^\perp.
\end{align*}
On the other hand, $[\mathfrak{p},\mathfrak{p}]\subset\mathfrak{k}$. Hence, $[v,w]=0$ and therefore $w\in Z_\mathfrak{p}(v)$. As $v$ is regular we have $Z_\mathfrak{p}(v)=\mathfrak{a}_v$ and the claim follows.
\end{proof}
\begin{Corollary}\label{complements}
At every regular $p\in TM$ we have
$$
T_{p}TM=\mathcal{D}_p\oplus\Upsilon_p.
$$
\end{Corollary}
\noindent
The next question is, what happens at the non-regular points? Indeed, the foliation becomes singular at non-regular points. Through every non-regular point go more than one polysphere/ polydisc. This is expressed by the fact that the distributions $\Upsilon$ and $\mathcal{D}$ become lower dimensional on singular points. \\
\ \\

\noindent
\textbf{Hermitian symmetric spaces as coadjoint orbits.}
As corollary of Theorem \ref{characterization HSS} and Proposition \ref{center is circle} we can finally deduce the realization of Hermitian symmetric spaces as coadjoint orbits. The corollary is known to the experts, but as we could not find a reference the proof is included here.
\begin{Corollary}\label{HSS as coadjoint orbit}
Every Hermitian symmetric space can be realized as (co-)adjoint orbit. 
\end{Corollary}
\begin{proof}
We prove the compact case, the noncompact case follows by duality. As $C(K)$ is analytically isomorphic to the circle group there exists an element $z\in C(K)$ different from the unit. Now $K$ is a sub group of the centralizer $C_G(z)$ of $z$ in $G$. As the center of $G$ is trivial, we have $C_G(z)\neq G$ so $K$ coincides with the identity component of $C_G(z)$ by maximality of $K$. Denote by $Z$ a generator of $C(K)$. Clearly, on the one hand, $K\subset\mathrm{Stab}_G(Z)$ and on the other hand, $\mathrm{Stab}_G(Z)\subset C_G(z)$. Thus
$$
\mathrm{Stab}_G(Z)=K,
$$
as stabilizers of simple groups are connected and therefore we may identify $G/K$ with the (co-)adjoint orbit $O_Z$ of $G$ at $Z\in\mathfrak{g}$.
\end{proof}
\noindent
\textbf{Kähler structure.}
The last question that needs to be answered is, whether the KKS symplectic structure $\sigma$ of $O_Z\cong M$ complements the hermitian structure $(g,j)$ to a Kähler structure. As $\sigma$, $g$ and $j$ are $G$-invariant it is enough to check compatibility at $Z\in O_Z\subset\mathfrak{g}$, indeed for all $a,b\in \mathfrak{p}\cong T_Z O_Z$
$$
g_Z(j_Z a,b)=-B([Z,a],b)=-B(Z,[a,b])=\sigma_Z(a,b).
$$
By uniqueness of $g$ and $j$ we obtain the following theorem.
\begin{Theorem}
The $G$-invariant triple $(g,j,\sigma)$, defined on $\mathfrak{p}\cong T_Z M$ as
$$
g_Z(\cdot,\cdot):=-B(\cdot,\cdot),\ \ j_Z(\cdot):=[Z,\cdot]\ \ \text{and}\ \ \sigma_Z(\cdot,\cdot):=-B(Z,[\cdot,\cdot])
$$
is compatible and equivariantly extends to the up to scalar multiple unique invariant Kähler structure of $M\cong O_Z$.
\end{Theorem}
\noindent
The following Lemma will be useful for future calculations.
\begin{Lemma}\label{invariant Kähler structure}
    At any point $x\in M\cong O_Z\subset\mathfrak{g}$ the Kähler structure is given by
    $$
        g_x(v,w):=-B(v,w),\ \ j_x(v):=[x,v]\ \ \text{and}\ \ \sigma_x(v,w):=-B(x,[v,w]),
    $$
    for all $v,w\in T_xM\cong[x,\mathfrak{g}]\subset\mathfrak{g}$.
\end{Lemma}
\begin{proof}
    We prove the formulas for the metric and the complex structure, the formula for the symplectic form follows. Clearly there exists an element $g\in G$ such that $x=\mathrm{Ad}_g(Z)$. Using this we find for the metric
    $$
    g_x(v,w):=g_Z(\mathrm{Ad}_{g^{-1}}v,\mathrm{Ad}_{g^{-1}}w)=-B(\mathrm{Ad}_{g^{-1}}v,\mathrm{Ad}_{g^{-1}}w)=-B(v,w)
    $$
    as the Killing form is $\mathrm{Ad}_G$-invariant. For the complex structure we similarly find
    $$
    j_x(v):= \mathrm{Ad}_{g}\left( j_Z\left (\mathrm{Ad}_{g^{-1}}v\right)\right)=\mathrm{Ad}_{g}[Z,\mathrm{Ad}_{g^{-1}}v]=[\mathrm{Ad}_{g}Z,v]=[x,v].
    $$
\end{proof}
\ \\
\textbf{De Rham cohomolgy.}
We will quickly determine the second de Rham cohomology of Hermitian symmetric spaces of compact type. We did not find a proof in the literature, so we present what we learned from discussions with Maria Beatrice Pozzetti. 
\begin{Proposition}\label{H2 is 1 dimensional}
    Let $(M,g)$ be an irreducible Hermitian symmetric space of compact type and denote $\sigma\in \Omega^2(M)$ the corresponding invariant Kähler form. Then the second de Rham cohomology group is generated by $[\sigma]$, i.e.
    $$
    H^2_{\mathrm{dR}}(M,\R)\cong \R.
    $$
\end{Proposition}
\begin{proof}
    Denote by $\Omega_G^2(M)$ the set of $G$-invariant 2-forms on $M$. Every $\nu\in \Omega^2_G(M)$ is closed. This can be seen as follows. For any point $p\in M$ denote by $s_p$ the geodesic symmetry, then on the one hand
    $$
    s_p^*\nu=(-1)^2\nu=\nu
    $$
    on the other hand 
    $$
    s_p^*\dd\nu=(-1)^3\dd\nu=-\dd\nu,
    $$
    using that $G$ acts transitively and that $s_p\circ g\circ s_p^{-1}\in G$ for all $g\in G$ implies that $s_p^*\nu$ is also $G$-invariant. In total this means
    $$
    \dd\nu=-\dd\nu =0.
    $$
    Further every de Rham cohomology class $\alpha\in H_{\mathrm{dR}}^2(M)$ can be represented by an invariant form. Let $\mu$ be a $G$ bi-invariant probability measure on $G$.  We define the $G$-average of a 2-form $\eta \in \Omega^2(M)$  with respect to $\mu$ as 
    $$
    \bar{\nu}_p(v,w):= \int_{G} (g^{*}\nu)_p(v,w)\ \dd\mu(g) \ \text{ for all } v,w\in T_pM.
    $$
    Then for any closed 2-dimensional submanifold $\Sigma\subset M$ we have
    \begin{align*}
\bar\nu(\Sigma)&=\int_\Sigma\bar\nu=\int_\Sigma\left(\int_G(g^*\nu)_p(\partial_s \Sigma,\partial_t \Sigma)\dd\mu(g)\right)\dd s\dd t\\
&=\int_G\left(\int_\Sigma (g^*\nu)_p(\partial_s \Sigma,\partial_t \Sigma)\dd s\dd t\right)\dd\mu(g)=\int_G\nu(g(\Sigma))\dd\mu(g)\\
&=\int_G\nu(\Sigma)\dd\mu(g)=\nu(\Sigma),
    \end{align*}
    where we think of $\Sigma(s,t)$ as a parametrization of $\Sigma\subset M$.
     All that is left to do is to show that there is up to scalar multiple only one invariant 2-form. Take some $\nu\in \Omega_G^2(M)$, then there exists a $K$-invariant symmetric operator $A:\mathfrak{p}\to\mathfrak{p}$ satisfying
    $
    \nu(A\cdot,\cdot)=\sigma(\cdot,\cdot).
    $
    Then, by the same argument as in the proof of uniqueness of the invariant metric $g$, $A$ must be a multiple of the identity, because the representation of $K$ on $\mathfrak{p}$ is irreducible. 
\end{proof}
\noindent
\textbf{Momentum maps.}
We will now study the induced action of $G$ on the tangent bundle of $M=O_Z$. Actually, the action can also be seen as the restriction of the diagonal adjoint action of $G$ on $\mathfrak{g}\times\mathfrak{g}$ to
$$
TM=\left \lbrace (x,v)\in \mathfrak{g}\times\mathfrak{g}\ \vert \ x=\mathrm{Ad}_g(Z),\ v\in \mathrm{ann}(x)^\perp\right\rbrace,
$$
where $\mathrm{ann}(x) = \lbrace\eta\in \mathfrak{g}\ \vert\ [\eta,x]=0\rbrace$. In view of this, we see that evaluated at a point $(x,v)\in TM$ the induced vector field $a^\#$ takes the form
$$
a^\#_{(x,v)}=([a,x],[a,v])\in\mathfrak{g}\times\mathfrak{g}.
$$
This representation of $a^\#$ will be useful for what comes.
By construction the 1-forms $\lambda, \eta $ and $\tau$ (see \ref{dual 1-forms}) are invariant under isometries and consequently $G$-invariant. It therefore makes sense to ask whether there exist momentum maps for $\dd\lambda, \dd\eta$ and $\dd\tau$.

\begin{theorem}\label{momentum maps on TM}
The $G$-action on $TM$ is Hamiltonian with respect to the three symplectic\footnote{Actually, $\dd\tau$ is only symplectic outside the zero-section. Still its momentum map is globally defined.} forms $\dd \lambda, \dd \eta$ and $\dd \tau$. The momentum maps are respectively given by
$$
\mu_\lambda(x,v)=[x,v],\  \ \mu_\eta(x,v)=v,\ \ \mu_\tau(x,v)=-[[x,v],v]\ \ \ \forall (x,v)\in TM\subset\mathfrak{g}\times\mathfrak{g},
$$
using the identification of $\mathfrak{g}$ and $\mathfrak{g}^*$ via the Killing form $B$.
\end{theorem}
\begin{proof}
The maps are clearly equivariant, as commutators are. Further, for any $a\in\mathfrak{g}$, $(x,v)\in TM$ we have
\begin{align*}
\dd (B(\mu_\lambda(x,v),a))&=\dd (B([x,v],a))=\dd (B(v,[x,a]))\\
&=\dd (g(v,\dd\pi(a^\#))=\dd (\hat g(X,a^\#))=\dd (\lambda(a^\#))=\iota_{a^\#}\dd\lambda
\end{align*}
as $\lambda$ is invariant under the flow of $a^\#$ and as a consequence $\mathcal{L}_{a^\#}\lambda=0$. Analogously we find
$$
\dd (B(\mu_\eta(x,v),a))=\dd (B(v,a))=\dd (B([x,v],[x,a]))=\dd (g(H,a^\#))=\dd (\eta(a^\#))=\iota_{a^\#}\dd\eta
$$
and
\begin{align*}
\dd (B(\mu_\tau(x,v),a))&=\dd (B([[x,v],v],a))=-\dd (B([x,v],[v,a]))\\
&=-\dd (g(V,a^\#))=-\dd (\tau(a^\#))=-\iota_{a^\#}\dd\tau .
\end{align*}
\end{proof}

\noindent
\textbf{Polyspheres resp.\ polydiscs as suborbits.}
We want to give an explicit description of the polyspheres resp.\ polydiscs in Theorem \ref{polysphere theorem} as suborbits. First we fix some notation. We denote by $\Sigma_i$ the i-th factor of $\Sigma^r$. Every factor can be realized as an adjoint orbit in $\mathfrak{h}$. Here $\mathfrak{h}$ denotes either $\mathfrak{su}(2)$ in the compact case or $\mathfrak{sl}(2,\R)$ in the non-compact case. Denote $Z_i$ the up to sign unique element in the center of $\mathfrak{h}$ such that $\mathrm{ad}_{Z_i}^2=-\mathrm{id}$. Then $\Sigma_i\cong O_{Z_i}$ and the standard Kähler structure coincides with the Kähler structure obtain as in Theorem \ref{invariant Kähler structure} up to multiple.\\
\ \\
From the discussion above Theorem \ref{polysphere theorem}, we know that every polysphere resp.\ polydisc
$$
\iota_{p,q}:\Sigma^r\hookrightarrow M
$$
comes from integrating a subalgebra of $\mathfrak{g}$ isomorphic to $\mathfrak{su}(2)^r$ resp.\ $\mathfrak{sl}(2,\R)^r$. 
In particular for every embedding $\iota_{p,q}$ there is an injective Lie algebra homomorphism
$$
k_{p,q}:\mathfrak{h}^r\hookrightarrow\mathfrak{g}
$$
such that $(\dd\iota_{p,q})_{x_0}=k_{p,q}\vert_{T_{x_0}\Sigma^r}$, where $p=(x_0,v_0)\in T\Sigma^r$. By equivariance of the embedding (see Remark \ref{equivariance polysphere}, translation and reparametrization), we may restrict to $p=(Z_0:=\sum_i Z_i,v_0)$ for some $v_0\in T_{Z_0}\Sigma$ and $q=(Z,v)$ for some $v\in T_{Z}M$. We abbreviate $\iota:=\iota_{p,q}$ and $k:=k_{p,q}$.

\begin{Proposition}\label{polysphere as suborbit}
The affine linear map
$$ K: \mathfrak{h}^r\to\mathfrak{g};\ \xi\mapsto k(\xi)+Z-k(Z_0)$$ 
extends $\iota:\Sigma^r\to M$ equivariantly with respect to the adjoint action of $H^r\subset G$. This means the following diagrams commute
\begin{equation}\label{extension}\textbf{(extension)}\ \ \ \ \ \ \ \ \ \ \ \ \ \ \
    \begin{tikzcd}
\Sigma^r \arrow{r}{\iota} \arrow[swap]{d}{} & M \arrow{d}{} \\%
\mathfrak{h}^r \arrow{r}{K}& \mathfrak{g}
\end{tikzcd},
\end{equation} 
where the vertical arrows are the inclusions as coadjoint orbits and
\begin{equation}\label{equivariance}\textbf{(equivariance)}\ \ \ \ \ \ \ \ \ \ \ \
    \begin{tikzcd}
\mathfrak{h}^r \arrow{r}{K} \arrow[swap]{d}{\mathrm{Ad}_h} & \g \arrow{d}{\mathrm{Ad}_{R(h)}} \\%
\mathfrak{h}^r \arrow{r}{K}& \mathfrak{g}
\end{tikzcd}
\end{equation} 
for all $h\in H^r$, where $R: H^r\hookrightarrow G$ denotes as is Remark \ref{equivariance polysphere} the monomorphism of Lie groups that integrates $k$, i.e. $(\dd R)_e=k$. 
\end{Proposition}
\begin{proof}
To prove the Lemma we first need to show that $Z-k(Z_0)$ is invariant under $R(H^r)$, i.e.\ we need to show that for any $\xi\in\mathfrak{h}^r$ the following commutator vanishes
\begin{equation}\label{eq11}
    [k(\xi),Z-k(Z_0)]=0.
\end{equation} 
Denote the Cartan decomposition of $\mathfrak{h}^r$ as 
$
\mathfrak{h}^r=\mathfrak{k}_0\oplus\mathfrak{p}_0.
$
The map $k$ respects the Cartan decomposition, i.e.\ $k(\mathfrak{k}_0)\subset\mathfrak{k}$ and $k(\mathfrak{p}_0)\subset \mathfrak{p}$. 
To prove Eq.\ \eqref{eq11} we look at two cases $\xi\in \mathfrak{k}_0$ and $\xi\in\mathfrak{p}_0$.\\
\ \\
\textbf{Case $\xi\in \mathfrak{k}_0$:} We see that
$$
[k(\xi),Z-k(Z_0)]=[k(\xi),Z]-[k(\xi),k(Z_0)]=k[\xi,Z_0]=0,
$$
where the second equality uses $k(\xi)\in\mathfrak{k}$ and $Z$ in the center of $K$ and that $k$ is a Lie algebra homomorphism. The last equality uses that $Z_0$ is in the center of $\mathfrak{h}^r$. \\

\ \\
\textbf{Case $\xi\in \mathfrak{p}_0$:} As $\xi\in \mathfrak{p}_0\cong T_{Z_0}\Sigma^r$ we can use $k(\xi)=\dd\iota_{Z_0}(\xi)$, it follows that
$$
[k(\xi),Z-k(Z_0)]=[\dd\iota_{Z_0}(\xi),Z]-k([\xi, Z_0])=j_Z\dd\iota_{Z_0}(\xi)-\dd\iota_{Z_0}(j_{Z_0}\xi)=0,
$$
where we used again that $k$ is a Lie algebra homomorphism, that $j_Z=\mathrm{ad}_Z$ and $j_{Z_0}=\mathrm{ad}_{Z_0}$ and in the last equation that $\iota$ is holomorphic.\\
\ \\
Further $k:\mathfrak{h}^r\hookrightarrow\g$ is $H^r$-equivariant as $k=(\dd R)_e$ and $R$ satisfies
$$
R(\mathrm{Ad}_h(\tilde h))=R(h\tilde h h^{-1})= R(h) R(\tilde h) R(h)^{-1}=\mathrm{Ad}_{R(h)}R(\bar h)\ \ \ \forall\ h,\tilde h\in H^r,
$$
as it is a Lie group homomorphism. \\
\ \\
Equivariance of $k$ and invariance of $Z-k(Z_0)$ imply equivariance of $K$, i.e. diagram \eqref{extension}. 
Last we need to check if $K$ extends $\iota$. As $K$ and $\iota$ are equivariant and $H^r$ acts transitively on $\Sigma^r$ it is enough to check this at one point $p\in \Sigma^r$. We choose $p=Z_0$ and find
$$
K(Z_0)=k(Z_0)+Z-k(Z_0)=Z=\iota(Z_0).
$$
This proves diagram \eqref{equivariance} and thus finishes the proof of the Proposition.
\end{proof}

\ \\
\noindent
\textbf{Hyperkähler structure of the tangent bundle.}
The hyperkähler structure of cotangent bundles of Hermitian symmetric spaces was described explicitly by Biquard and Gauduchon \cite{BG21}. The case of constant holomorphic sectional curvature\footnote{Constant holomorphic sectional curvature is the same as rank one Hermitian symmetric space.} is one of the very first hyperkähler structures ever described by Calabi \cite{C79}. The clue to extend Calabi's formulas to higher rank spaces is using spectral functions on the self-adjoint operator
$$
jR_{jv,v}: T_xM\to T_xM;\ \ \ w\mapsto jR(jv,v)w
$$
for some $v\in T_xM$. Further recall that
$$
U_\rho^2 M:=\lbrace (x,v)\in TM\vert \  \vert g_x(jR_{jv,v}w,w) \vert <\rho^2\Vert w\Vert^2\ \ \forall w\in T_xM\rbrace
$$
denotes the neighborhood of the zero-section with the absolute value of the holomorphic-bisectional curvature bounded by $\rho^2/\Vert v\Vert^2$.

\begin{theorem}[\cite{BG21}]\label{hyperkähler HSS}
Let $M$ be a Hermitian symmetric space, then
there is a unique $G$-invariant hyperkähler metric on $TM$ in the compact and on resp.\ $U_{1}M$ in the non-compact case, such that the Kähler form compatible with $K=j\ominus j$ is given by
$$
\omega_K = \pi^*\sigma + \dd\dd^c\nu,
$$
with
$$
\nu((x,v)) = g_x(F(jR_{jv,v})v,v),\ \ F(y) =\frac{1}{y}\left(\sqrt{1 + y} - 1 - \ln\frac{1+\sqrt{1+y}}{2}
\right ).
$$
If $M$ is of non-compact type the hyperkähler metric is incomplete.
\end{theorem}
\noindent
Now that we have an explicit formula for the symplectic form $\omega_K$ we can show that the induced $G$-action on $TM$ is actually Hamiltonian.
\begin{Proposition}\label{moment map hyperkählefr}
The symplectic action of $G$ on $(TM, \omega_K)$ resp.\ $(U_{1}M, \omega_K)$ is Hamiltonian with moment map
$$
\mu_K(x,v):=-[v,j\tilde F(jR_{jv,v})v]+x
$$
\end{Proposition}
\begin{proof}The map $\mu_K$ is equivariant as commutators and $jR_{jv,v}$ are. We check by a direct computation, that
\begin{align*} 
   \iota_{a^\#} (\dd \dd^c \nu)&=-\dd (\hat g((j\tilde F(jR_{jv,v})v)^\VV,  a^\#))=-\dd (g(j\tilde F(jR_{jv,v})v, P_\VV( a^\#)))\\
   &=\dd (B(j\tilde F(jR_{jv,v})v, [a,v]))
   =\dd (B([v,j\tilde F(jR_{jv,v})v], a)),
\end{align*}
where in the first equation we used that $\mathcal{L}_{a^\#}\dd^c\nu=0$.
Now recall from section \ref{sec1.1} the inclusion map is a moment map with respect to the symplectic form $\sigma$, thus
$$
\iota_{a^\#}\omega_K=\iota_{a^\#} \dd \dd^c \nu+\iota_{a^\#} \pi^*\sigma=\dd B(-[v,j\tilde F(jR_{jv,v})v]+x, a)= \dd(\mu_K,a)
$$
follows.
\end{proof}

\ \\
\textbf{Ma\~{n}\'{e}'s critical value.}
The symplectomorphisms of Theorems \ref{twisted to hyperkähler} and \ref{hyperkähler to constant} are on the hyperkähler side defined precisely on the neighborhood of the zero-section where the hyperkähler structure exists (see Thm.\ \ref{hyperkähler HSS}). The image of the symplectomorphism in Theorem \ref{hyperkähler to constant} is the maximal set where the (symplectic) form $\dd\tau/2-\pi^*\sigma$ is non-degenerate. So we asked ourselves what meaning does the domain of the symplectomorphism in Theorem \ref{twisted to hyperkähler} have. The answer is related to the Ma{\~n}{\'e} critical value of Hermitian symmetric spaces $(M,g,\sigma)$. 

\begin{Proposition}
    The Ma{\~n}{\'e} critical value of a locally Hermitian symmetric space is infinite, when the covering space is of compact type and $r/2$, when the covering space is of non-compact type and has rank $r$.
\end{Proposition}
\begin{proof}
First note that the first part of the proposition is clear as the invariant symplectic form $\sigma$ is not weakly exact if the universal cover is of compact type.\\
If the universal cover is of non-compact type we can, using the polydisk theorem \ref{polysphere theorem}, adapt the computation of the Ma{\~n}{\'e} critical value of complex hyperbolic manifolds from \cite[Sec.\ 5.2]{CFP10}. Recall from Definition \ref{Mane value via forms} that the Ma{\~n}{\'e} critical value is 
$$
c(M,g,\sigma):=\frac{1}{2}\inf_\theta \sup_{x\in\hat M}\Vert {}^g\theta_x\Vert^2,
$$
where $\hat M$ denotes the universal cover of $M$, the infimum is taken over primitives of $ \sigma$ and ${}^g\theta$ denotes the metric dual of $\theta$. Thus in order to bound the Ma{\~n}{\'e} critical value from above we need to find a primitive of $\sigma$. It is well known that the invariant Kähler form $\hat \sigma$ on $\hat M$ is exact and splits along the polydiscs. This means
$$
\Vert {}^g\theta_x\Vert=\max_{\Vert v\Vert =1} \theta_x(v)=\max_{\Vert v\Vert =1} \sum_{i=1}^r\theta_x(v_i)=\frac{1}{\sqrt{r}}\sum_{i=1}^r\max_{\Vert v_i\Vert =1}\dd^c\nu_x(v_i)=\sqrt{r}\Vert {}^g\dd^c\nu\Vert.
$$
Here, the index denotes the splitting along a polydisc that is tangent to $v$. For the third equality we use that the Kählerpotential in \cite{Whd05} restricted to each factor $\ch^1\cong \lbrace x^2+y^2<1\rbrace $ is given by $\nu(x,y)=-\log(1-x^2-y^2)$. Hence, the corresponding primitive restricted to a factor is given by $\dd^c\nu=\frac{2y\dd x-2x\dd y}{1-x^2-y^2}$. A quick computation shows that $\Vert {}^g\dd^c\nu\Vert=x^2+y^2$, hence $\sup_{(x,y)\in\ch^1}\Vert {}^g\dd^c\nu\Vert=1$ and therefore $c(M,g,\sigma)\leq r/2$.\\
\ \\
We will now follow the proof of Lemma 6.11 in \cite{CFP10} to find a lower bound. Consider the family of closed curves $\gamma_R: [0, T]\to M$ inside a polydisc of the form $(\gamma_1,\ldots,\gamma_r)$, where $\gamma_i: [0, T]\to \C\mathrm{H}^1$ parametrizes a geodesic circle of radius $R$ with speed $\vert \dot\gamma_i\vert=\sqrt{2k/r}$. We know that the primitive $\theta$ of $\sigma$ on $M$ pulls back to the primitive $\sum_i\dd^c\nu$ on the polydisc, where we identified $\ch^1$ with the Poincaré disk. Now we can compute
\begin{align*}
\mathcal{A}_{L+k}(\gamma)&=\sum_{i=1}^r\int_0^T\left( \frac{1}{2}\vert \dot \gamma_i(t)\vert^2+\dd^c\nu(\dot\gamma_i)+\frac{k}{r}\right)\dd t\\
&=\sum_{i=0}^r\left( \int_0^T \frac{2k}{r} \dd t-\int_{D_R}\frac{4\dd x\wedge \dd y}{(1-x^2-y^2)^2}\right )\\
&=r\left(\sqrt{\frac{2k}{r}}l-A\right),
\end{align*}
where we used that $T=l\sqrt{r/2k}$ and $l$ denotes the hyperbolic circumference and $A$ the hyperbolic area of a geodesic disc $D_R$ of Radius $R$. We plug in $l=2\pi\sinh(R)$ and $A=2\pi(\cosh(R)-1)$ to find
$$
k<\frac{r}{2}\ \ \ \Rightarrow\ \ \ \mathcal{A}_{L+k}(\gamma_R)\to-\infty \ \text{for}\ R\to\infty.
$$
By the alternative description of the Ma{\~n}{\'e} critical value via the action functional (see section \ref{Mane via action}) we have $r/2\leq c(M,g,\sigma)$.
\end{proof}

\section{Symplectic identifications}\label{symplectomorphisms}

In this section we prove Theorem \ref{twisted to hyperkähler}, Theorem \ref{hyperkähler to constant} and Theorem \ref{diagonal}. The idea is that the Hamiltonian $G$-actions on the manifolds involved have coisotropic orbits. Hence, intertwining these $G$-actions as in the following lemma yields symplectic identifications. The same strategy was already applied in the proof of \cite[Thm.\ C]{proj23}.

\begin{Lemma}[\cite{proj23}, Lem.\ 2.1]\label{momtri}
Assume we have two symplectic manifolds $(N_1, \omega_1)$ and $(N_2, \omega_2)$ with Hamiltonian $G$-actions. Denote by $\mu_i: N_i\to\mathfrak g^*$ for $i=1,2$ their moment maps. If $\phi: N_1\to N_2$ is an equivariant smooth bijection such that 
\[
\begin{tikzcd}
N_1 \arrow[dr,"\mu_1"] \arrow[rr,"\phi"] && N_2\arrow[dl,"\mu_2"]  \\
 & \mathfrak{g}^* 
\end{tikzcd}
\]
commutes and the distribution $\mathcal{D}\subset TN_1$ tangent to the $G$-orbits admits a complement $\Upsilon$ that is isotropic for both symplectic forms $\omega_1$ and $\phi^*\omega_2$, then $\phi$ is actually a symplectomorphism i.e.\ $\phi^*\omega_2=\omega_1$. 
\end{Lemma}
\noindent
The key ingredient in order to study tangent bundles of higher rank Hermitian symmetric spaces is the (singular) foliation in terms of polyspheres resp.\ polydiscs. Recall from the discussion below Definition \ref{regularpoint} that a regular point $(x,v)\in TM$ determines a unique polysphere resp.\ polydisc through $x$ tangent to $v$. Now the idea is, that our symplectomorphism coincides along the polyspheres resp.\ polydiscs with the symplectomorphisms discussed in the introduction. Inspired by Biquard--Gauduchon \cite{BG21}, we formalize this idea by promoting the functions $a,b,c_1,c_2$ (see \eqref{b},\eqref{a},\eqref{c}) to spectral functions for the self-adjoint operator
$$
jR_{jv,v}: T_xM\to T_xM;\ \ \ w\mapsto jR(jv,v)w
$$
for any $(x,v)\in TM$. Note that $jR_{jv,v}$ is diagonal along polyspheres resp.\ polysdiscs, e.g.
$$
\left (jR_{jv,v}\right)\vert_{T_x\Sigma^r}= \begin{pmatrix} 
    \pm \vert v_1\vert ^2 & \dots  & 0\\
    \vdots & \ddots & \vdots\\
    0 & \dots  & \pm \vert v_r\vert^2
    \end{pmatrix}
$$
for any $\Sigma^r$ through $x$, where $v_i$ is the projection of $v$ to the ith factor of $\Sigma^r$ and $\Sigma\in\lbrace\cp^1,\C\mathrm{H}^1\rbrace$ depending on the type of $M$.
This also implies that $(D_\rho\Sigma)^r = U_{\rho^2}(\Sigma^r)$.\\
\ \\
\noindent
We follow the same strategy three times to prove our theorems.
\begin{itemize}
    \item[(1)] Use spectral functions on $jR_{jv,v}$ to find a candidate $\phi$.
    \item[(2)] Show that $\phi$ is an equivariant diffeomorphism.
    \item[(3)] Show that the moment triangle commutes.
    \begin{itemize}
        \item[(a)] Reduce the general case to the case of polyspheres resp.\ polydiscs.
        \item[(b)] Show that the moment triangle commutes for polyspheres resp.\ polydiscs.
    \end{itemize}
    \item[(4)] Show that $\Upsilon=\mathrm{span}_\R\lbrace Y_1,\ldots, Y_r\rbrace$ (see Cor.\ \ref{complements}) is an isotropic complement of $\mathcal{D}$ for both $\omega_1$ and $\phi^*\omega_2$ and use Lemma \ref{momtri} to conclude the proof.
\end{itemize}
Note that in all our cases $N_1$ is a neighborhood of the zero-section of the tangent bundle of a Hermitian symmetric space. Hence, step (3) and (4) make sense.

\subsection{Proof of Theorem \ref{twisted to hyperkähler}}
Recall Theorem \ref{twisted to hyperkähler} from the introduction. 
\begin{theorem} Let $M$ be an irreducible Hermitian symmetric space of compact type, then there exists an equivariant symplectomorphism
$$
\phi: (TM,\omega_\sigma)\rightarrow (TM, \omega_K).
$$
If $M$ is of non-compact type the symplectomorphism exists only on a neighborhood of the zero-section, namely
$$
\phi: (U_{1}M,\omega_\sigma)\rightarrow (U_{1}M, \omega_K).
$$
\end{theorem}

\noindent
\textbf{Step (1).}
We define the symplectomorphism $\phi$ using spectral functions of $jR_{jv,v}$. Recall that in the 2-dimensional case Eq.\ \eqref{b} we have
$$
\phi(x,v)=\left(\exp_x(b(\kappa r^2)jv), P_\gamma v(1)\right),
$$
with
$$
    b(y)=\frac{\arctan(\sqrt{y})}{\sqrt{y}}.
$$
We promote $b$ to a spectral function of $jR_{jv,v}$ to obtain
$$
\phi(x,v):=\left(\exp_x(b(jR_{jv,v})jv), P_\gamma v(1)\right).
$$
Observe that $\phi$ is smooth as $b$ is smooth. Further it is defined whenever the eigenvalues of $jR_{jv,v}$ are bounded from below by $-1$. For compact type Hermitian symmetric spaces this holds on the whole tangent bundle $TM$, for non-compact type Hermitian symmetric spaces this holds on $U_{1}M\subset TM$\\

\noindent
\textbf{Step (2).}
We want to show that $\phi$ is an equivariant diffeomorphism. We start with equivariance.
\begin{Lemma}
The map $\phi$ is equivariant under the action of the isometry group.
\end{Lemma}
\begin{proof}
All objects, i.e.\ metric, curvature, exponential map are invariant under the action of isometries thus also $\phi$ is. Explicitly let $I:M\to M$ be an isometry, then
\begin{align*}
\phi(\dd I(x,v))&=\left(\exp_{I(x)}(b(jR_{j\dd I_x v,\dd I_xv})j\dd I_x v), P_{\gamma} \dd I_x v(1)\right) \\
&=\left(\exp_{I(x)}(\dd I_x b(jR_{j v,v})j v),\dd I_x P_\gamma  v(1)\right)\\
&=\left(I(\exp_{x}( b(jR_{j v,v})j v),\dd I_x  P_\gamma  v(1)\right)\\
&=\dd I(\phi(x,v)).
\end{align*}
\end{proof}
\begin{Lemma}
The map $\phi$ is a diffeomorphism and $\phi^{-1}$ is defined on $TM$ resp.\ $U_{1}M$ in the compact resp.\ non-compact case.
\end{Lemma}
\begin{proof}
Also in analogy to the constant curvature case one can explicitly give an inverse $\phi^{-1}$ where
$$
\phi^{-1}(x,v)=\left(\exp_x(-b(jR_{jv,v})jv), P_\gamma v(1)\right).
$$
The inverse is well defined whenever the eigenvalues of $jR_{jv,v}$ are bounded from below by $-1$, i.e.\ on $TM$ resp.\ $U_{1}M$ in the compact resp.\ non-compact case. In particular it is well-defined on the image of $\phi$ namely $TM$ resp.\ $U_{1}M$. The inverse is smooth as $b$ is smooth. 
\end{proof}
\noindent
\textbf{Step (3a):}\\
Recall the moment map $\mu_1(x,v):=[x,v]+x$ for $\omega_\sigma$ (see \ref{momentum maps on TM}) and the moment map $\mu_2:=\mu_K$ for $\omega_K$ (see \ref{moment map hyperkählefr}). We need to assure that $\phi$ intertwines the moment maps, i.e. $\mu_1= \mu_2\circ \phi$.
This is difficult to see if one considers the full Hermitian symmetric space, but relatively easy to prove for polyspheres resp.\ polydiscs. We will show now that we can actually reduce the general case to polyspheres resp.\ polydiscs. 
\begin{Lemma}
    The diagrams
\begin{equation}\label{diag1}
    \begin{tikzcd}
(T\cp^1)^r \arrow{r}{\dd \iota} \arrow[swap]{d}{\phi\times\ldots\times \phi} & TM \arrow{d}{\phi} \\%
(T\cp^1)^r\ \ \arrow{r}{\dd\iota}& \ \ TM
\end{tikzcd}
\end{equation} 
in the compact case and
\begin{equation}\label{diag2}
    \begin{tikzcd}
(D_{1}\C\mathrm{H}^1)^r \arrow{r}{\dd \iota} \arrow[swap]{d}{\phi\times\ldots\times \phi} & U_{1}M \arrow{d}{\phi} \\%
(D_{1}\C\mathrm{H}^1)^r\ \ \arrow{r}{\dd\iota}& \ \ U_{1}M
\end{tikzcd}
\end{equation} 
in the non-compact case commute. 
\end{Lemma}
\begin{proof}
    The self adjoint endomorphism $jR_{jv,v}$ restricts to the tangent spaces of the poly- spheres/disc. As the embedding is complex totally geodesic it follows that $\phi$ also restricts to the copies of $T\Sigma^r$. Further $jR_{jv,v}$ is diagonal with respect to the splitting of $T\Sigma^r$ as product $T\Sigma\times\ldots\times T\Sigma$ and therefore the diagram holds.
\end{proof}
\noindent
Next we include the moment maps into the diagrams \eqref{diag1} and \eqref{diag2}. For the compact case we look at
\begin{equation}
\begin{tikzcd}
(T\cp^1)^r \arrow[dddd,"\phi\times\ldots\times\phi"] \arrow[rrrrrr,hookrightarrow,"\dd \iota"]\arrow[ddrr,"\mu_1\times\ldots\times\mu_1"] &&&&&& TM \arrow[dddd,"\phi"]\arrow[ddll, "\mu_{1}"]\\
&&&\kreis{1}&&&\\
& \kreis{2} & \mathfrak{su}(2)^r\arrow[rr,hookrightarrow, "K"] && \mathfrak{g} & \kreis{4} & \\
&&&\kreis{3}&&&\\
(T\cp^1)^r \arrow[uurr, "\mu_2\times\ldots\times\mu_2"] \arrow[rrrrrr,hookrightarrow,"\dd \iota"]&&&&&& TM \arrow[uull, "\mu_2"]
\end{tikzcd}
\end{equation}
and for the non-compact case we look at
\begin{equation*}
\begin{tikzcd}
(D_{1}\C\mathrm{H}^1)^r\arrow[dddd,"\phi\times\ldots\times\phi"] \arrow[rrrrrr,hookrightarrow,"\dd \iota"]\arrow[ddrr,"\mu_1\times\ldots\times\mu_1"] &&&&&& U_{1}M \arrow[dddd,"\phi"]\arrow[ddll, "\mu_1"]\\
&&&\kreis{1}&&&\\
& \kreis{2} & \mathfrak{sl}(2,\R)^r\arrow[rr,hookrightarrow, "K"] && \mathfrak{g} & \kreis{4} & \\
&&&\kreis{3}&&&\\
(D_{1}\C\mathrm{H}^1)^r \arrow[uurr, "\mu_2\times\ldots\times\mu_2"] \arrow[rrrrrr,hookrightarrow,"\dd \iota"]&&&&&& U_{1}M \arrow[uull, "\mu_2"]
\end{tikzcd},
\end{equation*}
where $K: \mathfrak{h}^r\hookrightarrow \mathfrak{g}$ is an affine embedding we will specify below.
The idea is that commutativity of $\ \kreis{4}$ follows from commutativity of $\ \kreis{1}-\ \kreis{3}$. As all maps are equivariant and the embeddings are well behaved under $G$ as explained in Remark \ref{equivariance polysphere}, we may assume $\iota=\iota_{p,q}$ for $p=(Z,v)$ and $q=(Z_0=\sum_{i=1}^rZ_i,v_0)$ for some elements $v\in T_ZM$ resp.\ $v_0\in T_{Z_i}\Sigma_i$. Here, $Z_i$ denotes the up to sign unique element in the center of $\mathfrak{h}_i$ such that $O_{Z_i}\cong \Sigma$ and the Kähler structure from Theorem \ref{invariant Kähler structure} coincides with the standard Kähler structure. The index indicates the factor in the product.\\
\ \\
We are now exactly in the setup of Proposition \ref{polysphere as suborbit}. We define $K$ to be
$$
K: \mathfrak{h}^r\hookrightarrow\mathfrak{g};\ h\mapsto k(h)+(Z-k(Z_0)).
$$
The affine embedding $K$ is the same as in Proposition \ref{polysphere as suborbit}.
\begin{Lemma}\label{diagram affine hyperkähler}
The sub diagrams $\ \kreis{1}$ and $\ \kreis{3}$ commute with this choice of affine embedding.
\end{Lemma}

\begin{proof}
 We start with $\ \kreis{1}$. Take $(y,w)\in (T\cp^1)^r$ resp.\ $(y,w)\in (D_{1}\C\mathrm{H}^1)^r$, then 
\begin{align*}
    K(\mu_1(y,w))&=k(\mu_\lambda(y,w)+\mu_\sigma(y,w))+(Z-k(Z_0))\\
    &=k([y,w])+K(y) \overset{\ast}{=}[k(y),k(w)]+K(y) \\
    &\overset{\ast\ast}{=}[K(y), k(w)]+K(y) \overset{\ast\ast\ast}{=}[\iota(y), \dd\iota_{y}(w)]+\iota(y)\\
    &=\mu_1(\dd\iota(y,w)).
\end{align*}
The first two and the last equations are just plugging in definitions. Equation $\ast$ uses that $k$ is a Lie algebra homomorphism, $\ast\ast$ uses that $[Z-k(Z_0),k(w)]=0$ by \eqref{eq11} and $\ast\ast\ast$ uses that $K$ extends $\iota$, $K\vert_{\Sigma^r}=\iota$ (see Prop.\ \ref{polysphere as suborbit}).\\
\ \\
Similarly we can compute $\ \kreis{3}$. Again take $(y,w)\in (T\cp^1)^r$ resp.\ $(y,w)\in (D_{1}\C\mathrm{H}^1)^r$, then 
\begin{align*}
    K((\mu_2 (y,w))&=k([w,j\tilde F(-jR_{jw,w})w]+y)+(Z-k(Z_0))\\
    &=k([w,j\tilde F(-jR_{jw,w})w])+K(y)\\
    &\overset{\ast}{=}[k (w),k(j\tilde F(-jR_{jw,w})w)]+K(y)\\
    &\overset{\ast\ast}{=}[\dd\iota_y w,\dd\iota_y(j\tilde F(-jR_{jw,w})w])]+\iota(y)\\
    &\overset{\ast\ast\ast}{=}[\dd\iota_y w,j\tilde F(-jR_{j\dd\iota_y w,\dd\iota_y w})\dd\iota_y w])]+\iota(y)\\
    &=\mu_2(\dd\iota(y,w)).
\end{align*}
The first two and the last equations are just plugging in definitions. Equation $\ast$ uses that $k$ is a Lie algebra homomorphism, $\ast\ast$ uses that $K$ extends $\iota$, $K\vert_{\Sigma^r}=\iota$ (see Prop.\ \ref{polysphere as suborbit}) and $\ast\ast\ast$ uses that $\iota$ is a holomorphic isometry.
\end{proof}
\noindent
Observe, that if we can now show that the moment map triangle commutes in the two dimensional case, commutativity of $\ \kreis{2}$ and thus commutativity of $\ \kreis{4}$ follows. \\
\ \\
\noindent
\textbf{Step (3b):}\\
We reduced the problem to the 2-dimensional case, thus $\Sigma\in\lbrace \cp^1,\C\mathrm{H}^1\rbrace$. In this case the isometries act transitively on the unit-sphere subbundle of $T\Sigma$, by equivariance it is therefore enough to show that
\begin{equation}
   \mu_2(\phi(Z_0, rv_0))=\mu_1(Z_0,rv_0) 
\end{equation}
for some fixed $(Z_0,v_0)\in T\Sigma$ and arbitrary $r\geq 0$\footnote{Attention! This $r$ has nothing to do with the rank, it is the norm of $v_0$.}.\\
The geodesic starting at $x\in M\subset \mathfrak{g}$ in direction $v\in T_xM\subset \mathfrak{g}$ is given by
\begin{equation}\label{eq7}
\gamma(t)=e^{tjv}xe^{-tjv}
\end{equation}
as then 
$$
\dot\gamma(0)=[jv,x]=-j^2v=v.
$$
Observe that
in the case of constant curvature surfaces all eigenvalues of $jR_{jv,v}$ are identically $y=\kappa r^2$. We compute
\begin{align*}
\tilde F(y)&:=F'(y)y+F(y)\\
&=\left ( \frac{1}{y^2}\ln\left (\frac{1+\sqrt{1+y}}{2}\right)-\frac{1}{2y(\sqrt{1+y}+1)}\right)y-F(y)\\
&=\frac{1}{\sqrt{1+y}+1}.
\end{align*}
Thus $\mu_2=\mu_K$ is in the case of surfaces given by 
$$
\mu_2(x,v)=\frac{\kappa r^2}{\sqrt{1+\kappa r^2}+1}x+x=\left(\frac{\kappa r^2+1+\sqrt{1+\kappa r^2}}{\sqrt{1+\kappa r^2}+1}\right)x=\sqrt{1+\kappa r^2}x.
$$
Now we can explicitly realize $\cp^1$ and $\ch^1$ as (co-)adjoint orbits and check the moment map condition.\\
\ \\
\noindent
\textbf{\textit{The case }$M=\cp^1$ :}
Here $G=\mathrm{SU}(2)$ and the Lie-algebra is
$$
\mathfrak{su}(2)=\left \langle a_1:=\frac{1}{2}\begin{pmatrix}
        i & 0\\
        0  & -i \\
    \end{pmatrix},\ a_2:=\frac{1}{2}\begin{pmatrix}
        0 & i\\
        i  & 0 \\
    \end{pmatrix},\ a_3:=\frac{1}{2}\begin{pmatrix}
        0 & 1\\
        -1  & 0 \\
    \end{pmatrix}\right \rangle
$$
The generators satisfy 
$$
[a_1,a_2]=-a_3,\ [a_3,a_1]=-a_2,\ [a_3,a_2]=a_1.
$$
We can identify $\cp^1$ as coadjoint orbit of $a_3$, i.e.\ $\cp^1\cong\mathcal{O}_{a_3}$. We need to show that $\phi$ intertwines the moment maps. The geodesic starting at $a_3$ in direction $-a_2=ja_1$ is 
$$
    \gamma(t)=e^{-ta_1}a_3e^{ta_1}=\frac{1}{2}\begin{pmatrix}
        0 & e^{-it}\\
        -e^{it}  & 0\\
    \end{pmatrix}.
$$
Here $\kappa=1$ and we can use Euler's formula $e^{ix}=\cos(x)+i\sin(x)$ and the identities
$$
\cos(\tan^{-1}(x))=\frac{1}{\sqrt{1+x^2}},\ \ \sin(\tan^{-1}(x))=\frac{x}{\sqrt{1+x^2}}
$$
in order to find the following expression
$$
\begin{pmatrix}
        0 & e^{-ibr}\\
        -e^{ibr}  & 0\\
    \end{pmatrix}=\cos\left(\tan^{-1}\left( r\right)\right)a_3-\sin\left(\tan^{-1}\left(r\right)\right)a_2=\frac{1}{\sqrt{1+r^2}}(a_3-ra_2).
$$
So in particular we have that:
\begin{align*}
    \mu_2(\phi(a_3, ra_1))&=\mu_2(\frac{1}{\sqrt{r^2+1}}(a_3-ra_2),ra_2)\\
    &=a_3-ra_2=[a_3,ra_1]+a_3\\
    &=\mu_1(a_3,ra_1).
\end{align*}
Which finishes the compact case.\\
\ \\
\noindent
\textbf{\textit{The case }$M=\C\mathrm{H}^1$ :}
Here $G=\mathrm{SU}(1,1)$ and the Lie-algebra is
$$
\mathfrak{su}(1,1)=\left \langle a_1:=\frac{1}{2}\begin{pmatrix}
        1 & 0\\
        0  & -1 \\
    \end{pmatrix},\ a_2:=\frac{1}{2}\begin{pmatrix}
        0 & 1\\
        1  & 0 \\
    \end{pmatrix},\ a_3:=\frac{1}{2}\begin{pmatrix}
        0 & 1\\
        -1  & 0 \\
    \end{pmatrix}\right\rangle
$$
The generators satisfy 
$$
[a_1,a_2]=a_3,\ [a_3,a_1]=-a_2,\ [a_3,a_2]=a_1.
$$
We can identify $\C\mathrm{H}^1$ as coadjoint orbit of $a_3$, i.e.\ $\C\mathrm{H}^1\cong\mathcal{O}_{a_3}$. We need to show that $\phi$ intertwines the moment maps. The geodesic starting at $a_3$ in direction $-a_2=ja_1$ is
\begin{align*}
    \gamma(t)=e^{-ta_1}a_3e^{ta_1}=\frac{1}{2}\begin{pmatrix}
        0 & e^{-t}\\
        -e^{t}  & 0\\
    \end{pmatrix}.
\end{align*}
Here $\kappa=-1$ and we can use the identities $e^{x}=\cosh(x)+\sinh(x)$ and
$$
-\cosh(\tanh^{-1}(x))=\frac{1}{\sqrt{1-x^2}},\ \ \sinh(\tanh^{-1}(x))=\frac{x}{\sqrt{1-x^2}}
$$
in order to find the following expression
$$
\begin{pmatrix}
        0 & e^{-br}\\
        -e^{br}  & 0\\
    \end{pmatrix}=\cosh\left(\tanh^{-1}\left( r\right)\right)a_3-\sinh\left(\tanh^{-1}\left( r\right)\right)a_2=\frac{1}{\sqrt{1-r^2}}(a_3-ra_2).
$$
So in particular we have that:
\begin{align*}
    \mu_2(\phi(a_3, ra_1))&=\mu_2(\frac{1}{\sqrt{1-r^2}}(a_3-ra_2),ra_1)\\
    &=a_3-ra_2=a_3-ra_2=r[a_3,a_1]+a_3\\
    &=\mu_1(a_3,ra_1),
\end{align*}
which finishes the proof of the non-compact case.\\
\ \\
\noindent
\textbf{Step (4):}\\
The last condition we need to check in order to apply \ref{momtri} (and thus finish the proof of Thm. \ref{twisted to hyperkähler}) is the existence of a complement of $\mathcal{D}:=\lbrace a^\#\ \vert \ a\in\mathfrak{g}\rbrace\subset TTM$ that is isotropic with respect to both symplectic forms $\omega_1:=\omega_\sigma$ and $\phi^*\omega_2:=\phi^*\left (\omega_I\right)$. Recall that by Corollary \ref{complements} on the open dense set of regular points $\Upsilon=\mathrm{span}\lbrace Y_1,\ldots, Y_r\rbrace$ is a complement of $\mathcal{D}$ and observe that $\Upsilon$ is clearly isotropic for $\omega_1$ as it is contained in the vertical distribution. 
\begin{Lemma}\label{Lemma isotropic}
$\Upsilon$ is isotropic for $\phi^* \omega_2$.
\end{Lemma}
\begin{proof}
In the view of the foliation of $TM$ by tangent spaces of polyspheres resp.\ polydiscs $T\Sigma^r$, near a regular point it makes sense to look at the vector fields $Y_i$. We need to compute $\dd \phi (Y_i)$. As $\phi$ splits with respect to the product we conclude that $\dd \phi (Y_i)\in T \Sigma_i$. Further also $\omega_2$ splits with respect to the product $T\Sigma_1\times\ldots\times T \Sigma_r$ and therefore $T\Sigma_i$ and $T\Sigma_j$ are $\omega_2$-orthogonal for $i\neq j$. It follows that
$$
\phi^*\omega_2(Y_i,Y_j)=\omega_2(\dd \phi Y_i,\dd \phi Y_j)=0\ \ \ \forall i,j\in\lbrace 1,\ldots, r\rbrace.
$$
\end{proof}
\noindent
Using Lemma \ref{momtri} we find that $\dd\phi$ is symplectic on the open, dense set of regular points, but as $\phi$ is smooth this already implies that $\dd\phi$ is symplectic everywhere.
Thus this Lemma finishes the proof of Theorem \ref{twisted to hyperkähler}.

\subsection{Proof of Theorem \ref{hyperkähler to constant}}
Recall Theorem \ref{hyperkähler to constant} from the introduction.
\begin{theorem} Let $M$ be an irreducible Hermitian symmetric space of compact type, then there exists an equivariant symplectomorphism
$$
\phi: (TM,\omega_K)\rightarrow (TM, \dd\tau/2+\pi^*\sigma).
$$
If $M$ is of non-compact type the symplectomorphism exists only on a neighborhood of the zero-section, namely
$$
\phi: (U_{1}M,\omega_K)\rightarrow (U_{2}M, \dd\tau/2+\pi^*\sigma).
$$
\end{theorem}

\noindent
\textbf{Step (1):}\\
In analogy to the 2-dimensional case \eqref{a}, define
$$
\phi(x,v)=\left(x, e^{a(jR_{jv,v})}v\right),
$$
with
$$
    a(y)=\frac{1}{2}\ln \left( \frac{2}{y}(\sqrt{1+y}-1)\right).
$$
We promoted $a$ to a spectral function of $jR_{jv,v}$. Observe that $\phi$ is smooth as $a$ is smooth. Further it is defined whenever the eigenvalues of $jR_{jv,v}$ are bounded from below by $-1$. For compact type Hermitian symmetric spaces this holds on the whole tangent bundle $TM$, for non-compact type Hermitian symmetric spaces this holds on $U_{1}M\subset TM$.\\

\noindent
\textbf{Step (2):}\\
We want to show that $\phi$ is an equivariant diffeomorphism. We start with equivariance.
\begin{Lemma}
The map $\phi$ is equivariant under the action of the isometry group.
\end{Lemma}
\begin{proof}
All objects, i.e.\ metric, curvature, exponential map are invariant under the action of isometries thus also $\phi$ is. Explicitly let $I:M\to M$ be an isometry, then
\begin{align*}
\phi(\dd I(x,v))&=\left(I(x), e^{a(jR_{j\dd I_xv,\dd I_xv})} \dd I_x v\right) \\
&=\left(I(x), \dd I_x e^{a( jR_{jv,v})}  v\right)\\
&=\dd I(\phi(x,v)).
\end{align*}
\end{proof}
\begin{Lemma}
The map $\phi$ is a diffeomorphism and $\phi^{-1}$ is defined on $TM$ resp.\ $U_{1}M$ in the compact resp.\ non-compact case.
\end{Lemma}
\begin{proof}
One can explicitly give an inverse $\phi^{-1}$, where
$$
\phi^{-1}(x,v)=\left(x, e^{\bar a(jR_{jv,v})}v\right),
$$
where 
$$
\bar a(y)=\frac{1}{2}\ln \left( \frac{1}{y}\left (\left(\frac{y}{2}+1\right)^2-1\right)\right).
$$
The inverse is well defined whenever the eigenvalues of $jR_{jv,v}$ are bounded from below by $-4$, i.e.\ on $TM$ resp.\ $U_{2}M$ in the compact resp.\ non-compact case. In particular it is well-defined on the image of $\phi$ namely $TM$ resp.\ $U_{2}M$. The inverse is smooth as $a$ is smooth. 
\end{proof}
\noindent
\textbf{Step (3a):}\\
Recall the moment map $\mu_1:=\mu_K$ for $\omega_1:=\omega_K$ (see \ref{moment map hyperkählefr}) and the moment map $\mu_2(x,v):=-[[x,v],v]/2+x$ for $\omega_2:=\dd\tau/2+\pi^*\sigma$ (see \ref{momentum maps on TM}). We need to assure that $\phi$ intertwines the moment maps, i.e. $\mu_1= \mu_2\circ \phi$.
This is difficult to see if one considers the full Hermitian symmetric space, but relatively easy to prove for polyspheres resp.\ polydiscs. We will show now that we can actually reduce the general case to polyspheres resp.\ polydiscs. 
\begin{Lemma}
    The diagrams
\begin{equation}\label{diag4}
    \begin{tikzcd}
(T\cp^1)^r \arrow{r}{\dd \iota} \arrow[swap]{d}{\phi\times\ldots\times \phi} & TM \arrow{d}{\phi} \\%
(T\cp^1)^r\ \ \arrow{r}{\dd\iota}& \ \ TM
\end{tikzcd}
\end{equation} 
in the compact case and
\begin{equation}\label{diag5}
    \begin{tikzcd}
(D_{1}\C\mathrm{H}^1)^r \arrow{r}{\dd \iota} \arrow[swap]{d}{\phi\times\ldots\times \phi} & U_{1}M \arrow{d}{\phi} \\%
(D_{2}\C\mathrm{H}^1)^r\ \ \arrow{r}{\dd\iota}& \ \ U_{2}M
\end{tikzcd}
\end{equation} 
in the non-compact case commute.
\end{Lemma}
\begin{proof}
    The self adjoint endomorphism $jR_{jv,v}$ restricts to and is diagonal with respect to the splitting of $T\Sigma^r$ as product $T\Sigma\times\ldots\times T\Sigma$ and therefore the diagram holds.
\end{proof}
\noindent
Next we include the moment maps into the diagrams \eqref{diag4} and \eqref{diag5}. For the compact case we look at
\begin{equation}
\begin{tikzcd}
(T\cp^1)^r \arrow[dddd,"\phi\times\ldots\times\phi"] \arrow[rrrrrr,hookrightarrow,"\dd \iota"]\arrow[ddrr,"\mu_1\times\ldots\times\mu_1"] &&&&&& TM \arrow[dddd,"\phi"]\arrow[ddll, "\mu_{1}"]\\
&&&\kreis{1}&&&\\
& \kreis{2} & \mathfrak{su}(2)^r\arrow[rr,hookrightarrow, "K"] && \mathfrak{g} & \kreis{4} & \\
&&&\kreis{3}&&&\\
(T\cp^1)^r \arrow[uurr, "\mu_2\times\ldots\times\mu_2"] \arrow[rrrrrr,hookrightarrow,"\dd \iota"]&&&&&& TM \arrow[uull, "\mu_2"]
\end{tikzcd}
\end{equation}
and for the non-compact case we look at
\begin{equation*}
\begin{tikzcd}
(D_{1}\C\mathrm{H}^1)^r\arrow[dddd,"\phi\times\ldots\times\phi"] \arrow[rrrrrr,hookrightarrow,"\dd \iota"]\arrow[ddrr,"\mu_1\times\ldots\times\mu_1"] &&&&&& U_{1}M \arrow[dddd,"\phi"]\arrow[ddll, "\mu_1"]\\
&&&\kreis{1}&&&\\
& \kreis{2} & \mathfrak{sl}(2,\R)^r\arrow[rr,hookrightarrow, "K"] && \mathfrak{g} & \kreis{4} & \\
&&&\kreis{3}&&&\\
(D_{2}\C\mathrm{H}^1)^r \arrow[uurr, "\mu_2\times\ldots\times\mu_2"] \arrow[rrrrrr,hookrightarrow,"\dd \iota"]&&&&&& U_{2}M \arrow[uull, "\mu_2"]
\end{tikzcd},
\end{equation*}
where $K: \mathfrak{h}^r\hookrightarrow \mathfrak{g}$ is an affine embedding we will specify below. 
The idea is that commutativity of $\ \kreis{4}$ follows from commutativity of $\ \kreis{1}-\ \kreis{3}$. As all maps are equivariant and the embeddings are well behaved under $G$ as explained in Remark \ref{equivariance polysphere}, we may assume $\iota=\iota_{p,q}$ for $p=(Z,v)$ and $q=(Z_0=\sum_{i=1}^rZ_i,v_0)$ for some elements $v\in T_ZM$ resp.\ $v_0\in T_{Z_i}\Sigma_i$. Here $Z_i$ denotes the up to sign unique element in the center of $\mathfrak{h}_i$ such that $O_{Z_i}\cong \Sigma$ and the Kähler structure from Theorem \ref{invariant Kähler structure} coincides with the standard Kähler structure. The index indicates the factor in the product.\\
\ \\
We are now exactly in the setup of Proposition \ref{polysphere as suborbit}. We define $K$ to be
$$
K: \mathfrak{h}^r\hookrightarrow\mathfrak{g};\ h\mapsto k(h)+(Z-k(Z_0)).
$$
The affine embedding $K$ is the same as in Proposition \ref{polysphere as suborbit}.
\begin{Lemma}
The sub diagrams $\ \kreis{1}$ and $\ \kreis{3}$ commute with this choice of affine embedding.
\end{Lemma}

\begin{proof}
We already proved $\ \kreis{1}$ in Lemma \ref{diagram affine hyperkähler}. So we only need to compute $\ \kreis{3}$. Take $(y,w)\in (T\cp^1)^r$ resp.\ $(y,w)\in (D_{1}\C\mathrm{H}^1)^r$, then 
\begin{align*}
    K((\mu_2 (y,w))&=k([w,[y,w]]/2+y)+(Z-k(Z_0))\\
    &=k([w,[y,w]]/2)+K(y)\\
    &\overset{\ast}{=}[k (w),[k(y),k(w)]]/2-sK(y)\\
    &\overset{\ast\ast}{=}[k(w),[K(y),k(w)]]/2-sK(y)\\
    &\overset{\ast\ast\ast}{=}[\dd\iota_y(w),[\iota(y),\dd\iota_{y} w]]/2-s\iota(y)\\
    &=\mu_2(\dd\iota(y,w)).
\end{align*}
The first two and the last equations are just plugging in definitions. Equation $\ast$ uses that $k$ is a Lie algebra homomorphism, $\ast\ast$ uses that $[Z-k(Z_0),k(w)]=0$ by \eqref{eq11} and $\ast\ast\ast$ uses that $K$ extends $\iota$, $K\vert_{\Sigma^r}=\iota$ (see Prop.\ \ref{polysphere as suborbit}).
\end{proof}
\noindent
Observe that if we can now show that the moment map triangle commutes in the two dimensional case, commutativity of $\ \kreis{2}$ and thus commutativity of $\ \kreis{4}$ follows. \\
\ \\
\noindent
\textbf{Step (3b):}\\
We reduced the problem to the 2-dimensional case, thus $\Sigma\in\lbrace \cp^1,\C\mathrm{H}^1\rbrace$. In this case the isometries act transitively on the unit-sphere subbundle of $T\Sigma$, by equivariance it is therefore enough to show that
\begin{equation}
   \mu_2(\phi(Z_0, rv_0))=\mu_1(Z_0,rv_0) 
\end{equation}
for some fixed $(Z_0,v_0)\in T\Sigma$ and arbitrary $r\geq 0$\footnote{Attention! This $r$ has nothing to do with the rank, it is the norm of $v_0$.}. Recall that 
$$
\mu_1(x,v)=\sqrt{1+\kappa r^2}x\ \ \text{and}\ \ \mu_2(x,v)=[v,[x,v]]/2+x=(1+\kappa r^2/2)x 
$$
We check that indeed:
\begin{align*}
    \mu_2(\phi(x, v))&=\mu_2(x,e^{a(\kappa r^2)}v)=(1+\kappa r^2 e^{2a(\kappa r^2)}/2)x= \sqrt{1+\kappa r^2}x=\mu_1(x,v).
\end{align*}
\ \\
\noindent
\textbf{Step (4):}\\
The last condition we need to check in order to apply Lemma \ref{momtri} (and thus finish the proof of Thm. \ref{hyperkähler to constant}) is the existence of a complement of $\mathcal{D}:=\lbrace a^\#\ \vert \ a\in\mathfrak{g}\rbrace\subset TTM$ that is isotropic with respect to both symplectic forms $\omega_1:=\omega_I$ and $\phi^*\omega_2:=\phi^*\left (\dd\tau/2+\pi^*\sigma\right)$. Recall that by Corollary \ref{complements} on the open dense set of regular points $\Upsilon=\mathrm{span}\lbrace Y_1,\ldots, Y_r\rbrace$ is a complement of $\mathcal{D}$. We already showed in Lemma \ref{Lemma isotropic} that $\Upsilon$ is isotropic for $\omega_1$. Completely analogously it also follows that $\Upsilon$ is isotropic for $\phi^*\omega_2$.
Using Lemma \ref{momtri} we find that $\dd\phi$ is symplectic on the open, dense set of regular points, but as $\phi$ is smooth this already implies that $\dd\phi$ is symplectic everywhere and we finish the proof of Theorem \ref{hyperkähler to constant}.

\subsection{Proof of Theorem \ref{diagonal}}
Recall Theorem \ref{diagonal} from the introduction.
\begin{theorem} Let $M\cong G/K$ be an irreducible Hermitian symmetric space of compact type, then there exists a $G$-equivariant symplectomorphism
$$
\phi: (U_{2\sqrt{R}}M,\omega_{(1-R)\sigma})\rightarrow (M\times M\setminus \bar\Delta, \sigma\ominus R\sigma),
$$
where $\bar \Delta$ is fiberwise the cut locus of the base point and thus a finite union of complex submanifolds of complex codimension one.
\end{theorem}

\noindent
\textbf{Step (1):}\\
Similar to the 2-dimensional case \eqref{c}, we define the symplectomorphism $\phi$ promoting the analytic functions $c_1,c_2: (-2\sqrt{R},2\sqrt{R})\to (0,\infty)$ implicitly defined via the equations 
$$
   \begin{aligned}
    \sin(2c_1\sqrt{y})+R\sin(2c_2\sqrt{y})&=\sqrt{y}, \\
    \cos(2c_1\sqrt{y})-R\cos(2c_2\sqrt{y})&=1-R
   \end{aligned}
$$
to spectral functions of $jR_{jv,v}$. For a detailed discussion of the functions $c_1,c_2$ and their definition look at \cite{proj23}.
We define
$$
\phi(x,v):=\left(\exp_x(c_1(jR_{jv,v})jv), \exp_x(-c_2(jR_{jv,v})jv)\right).
$$
Observe that $\phi$ is smooth as $c_1,c_2$ are smooth.\\

\noindent
\textbf{Step (2):}\\
We want to show that $\phi$ is an equivariant smooth bijection. We start with equivariance.
\begin{Lemma}
The map $\phi$ is equivariant with respect to the induced action of the isometry group on $TM$ and the diagonal action on $M\times M$.
\end{Lemma}
\begin{proof}
All objects, i.e.\ metric, curvature, exponential map are invariant under the action of isometries thus also $\phi$ is. Explicitly let $I:M\to M$ be an isometry, then
\begin{align*}
\phi(\dd I(x,v))&=\left(\exp_{I(x)}(c_1(jR_{j\dd I_x v,\dd I_xv})j\dd I_x v), \exp_{I(x)}(-c_2(jR_{j\dd I_x v,\dd I_xv})j\dd I_x v)\right) \\
&=\left(\exp_{I(x)}(\dd I_x c_1(jR_{j v,v})j v),\exp_{I(x)}(-\dd I_x c_2(jR_{j v,v})j v)\right)\\
&=\left(I(\exp_{x}( c_1(jR_{j v,v})j v),I(\exp_{x}(-c_2(jR_{j v,v})j v)\right)\\
&=(I\times I)(\phi(x,v)).
\end{align*}
\end{proof}
\begin{Lemma}
The map $\phi$ is a smooth bijection.
\end{Lemma}
\begin{proof}
We can explicitly give an inverse. Recall from Theorem \ref{polysphere theorem} that any two points $a,b\in M$ lie on a totally geodesic polysphere $(\cp^1)^r$. We may represent $a=(a_1,\ldots,a_n), b=(b_1,\ldots,b_r)$ according to the product. The condition $(a,b)\in M\times M\setminus \bar \Delta$, implies that the points $(a_i,b_i)\in \cp^1\times \cp^1$ are never antipodal. As shown in \cite[Thm.\ C]{proj23} we can find a unique point $(x_i,v_i)\in D_{2\sqrt{R}}\cp^1$ such that $\phi(x_i,v_i)=(a_i,b_i)$. Since $\phi$ restricts and splits with respect to the polyspheres this shows that $(x,v)=\dd\iota(x_1,\ldots,x_r,v_1,\ldots,v_r)\in U_{2\sqrt{R}}M$ satisfies $\phi(x,v)=(a,b)$. Moreover, $(x,v)$ is clearly unique if there is only one polysphere through $a$ and $b$. If there are more polyspheres through $a$ and $b$ this means $(x,v)$ is not regular. Hence, some of the $v_i=0$ or equivalently some of the $a_i=b_i$. This means the map $\phi$ is constant in those directions and therefore the inverse does not depend on the choice of them.
\end{proof}
\noindent
\textbf{Step (3a):}\\
Recall the moment map $\mu_1(x,v):=[x,v]+x$ for $\omega_\sigma$ (see \ref{momentum maps on TM}) and the moment map $\mu_2(a,b):=a-Rb$ for $\omega_2:=\sigma\ominus R\sigma$. We need to assure that $\phi$ intertwines the moment maps, i.e. $\mu_1= \mu_2\circ \phi$.
This is difficult to see if one considers the full Hermitian symmetric space, but relatively easy to prove for polyspheres. We will show now that we can actually reduce the general case to polyspheres. 
\begin{Lemma}
    The diagrams
\begin{equation}\label{diag11}
    \begin{tikzcd}
(D_{2\sqrt{R}}\cp^1)^r \arrow{r}{\dd \iota} \arrow[swap]{d}{\phi\times\ldots\times \phi} & U_{2\sqrt{R}}M \arrow{d}{\phi} \\%
(\cp^1\times\cp^1\setminus \bar \Delta)^r\ \ \arrow{r}{\iota\times \iota}& \ \ (M\times M\setminus\bar \Delta)
\end{tikzcd}
\end{equation} 
commute. 
\end{Lemma}
\begin{proof}
    As the embedding $\iota$ is complex totally geodesic it follows that $jR_{jv,v}$ and $\phi$ restrict and are diagonal along the polyspheres, hence the diagram holds.
\end{proof}
\noindent
Next we include the moment maps into diagram \eqref{diag11},
\begin{equation}
\begin{tikzcd}
(D_{2\sqrt{R}}\cp^1)^r \arrow[dddd,"\phi\times\ldots\times\phi"] \arrow[rrrrrr,hookrightarrow,"\dd \iota"]\arrow[ddrr,"\mu_1\times\ldots\times\mu_1"] &&&&&& U_{2\sqrt{R}}M \arrow[dddd,"\phi"]\arrow[ddll, "\mu_{1}"]\\
&&&\kreis{1}&&&\\
& \kreis{2} & \mathfrak{su}(2)^r\arrow[rr,hookrightarrow, "K"] && \mathfrak{g} & \kreis{4} & \\
&&&\kreis{3}&&&\\
(\cp^1\times\cp^1\setminus\bar\Delta)^r \arrow[uurr, "\mu_2\times\ldots\times\mu_2"] \arrow[rrrrrr,hookrightarrow,"\dd \iota"]&&&&&& M\times M\setminus\bar\Delta \arrow[uull, "\mu_2"]
\end{tikzcd}.
\end{equation}
where $K: \mathfrak{h}^r\hookrightarrow \mathfrak{g}$ is an affine embedding we will specify below. 
The idea is that commutativity of $\ \kreis{4}$ follows from commutativity of $\ \kreis{1}-\ \kreis{3}$. As all maps are equivariant and the embeddings are well behaved under $G$ as explained in Remark \ref{equivariance polysphere}, we may assume $\iota=\iota_{p,q}$ for $p=(Z,v)$ and $q=(Z_0=\sum_{i=1}^rZ_i,v_0)$ for some elements $v\in T_ZM$ resp.\ $v_0\in T_{Z_i}\cp^1_i$. Here $Z_i$ denotes the up to sign unique element in the center of $\mathfrak{h}_i$ such that $O_{Z_i}\cong \cp^1$ and the Kähler structure from Theorem \ref{invariant Kähler structure} coincides with the standard Kähler structure. The index indicates the factor in the product.\\
\ \\
We are now exactly in the setup of Proposition \ref{polysphere as suborbit}. We define $K_s$ to be
$$
K_s(a)=k(a)+s(Z-k(Z_0)).
$$
The affine embedding $K_s$ is a variation of the affine embedding $K$ defined in Proposition \ref{polysphere as suborbit}.
\begin{Lemma}
The sub diagrams $\ \kreis{1}$ and $\ \kreis{3}$ commute with this choice of affine embedding.
\end{Lemma}

\begin{proof}
The proof of $\ \kreis{1}$ is analogous to Lemma \ref{diagram affine hyperkähler}. So we only compute $\ \kreis{3}$.
\begin{align*}
    K_s(\mu_2(a,b))&=K_s(a-sb)=k(a-sb)+(1-R)(Z-k(Z_0))\\
    &=k(a)+(Z-k(Z_0))-R(k(b)+(Z-k(Z_0)))\\
    &=K(a)-RK(b)=\iota(a)-R\iota(b)\\
    &=\mu_2(\iota\times\iota(a,b))
\end{align*}
\end{proof}
\noindent
Observe that if we can now show that the moment map triangle commutes in the two dimensional case, commutativity of $\ \kreis{2}$ and thus commutativity of $\ \kreis{4}$ follows. \\
\ \\
\noindent
\textbf{Step (3b):}\\
We reduced the problem to the 2-dimensional case, thus $M=\cp^1$. This was proved in \cite[Thm.\ 5.1]{proj23}.\\
\ \\
\noindent
\textbf{Step (4):}\\
The last condition we need to check in order to apply Lemma \ref{momtri} (and thus finish the proof of Thm. \ref{diagonal}) is the existence of a complement of $\mathcal{D}:=\lbrace a^\#\ \vert \ a\in\mathfrak{g}\rbrace\subset TTM$ that is isotropic with respect to both symplectic forms $\omega_1:=\omega_\sigma$ and $\phi^*\omega_2:=\phi^*\left (\sigma\ominus R\sigma)\right)$. Recall that by Corollary \ref{complements} on the open dense set of regular points $\Upsilon=\mathrm{span}\lbrace Y_1,\ldots, Y_r\rbrace$ is a complement of $\mathcal{D}$ and observe that $\Upsilon$ is clearly isotropic for $\omega_1$ as it is contained in the vertical distribution. 
Completely analogous to Lemma \ref{Lemma isotropic} it follows that $\Upsilon$ is isotropic for $\phi^*\omega_2$. Using Lemma \ref{momtri} we find that $\dd\phi$ is symplectic on the open, dense set of regular points, but as $\phi$ is smooth this already implies that $\dd\phi$ is symplectic everywhere.
Thus this Lemma finishes the proof of Theorem \ref{diagonal}.

\section{Computing some capacities}\label{capacities}
In this section we compute some capacities, namely prove Theorem \ref{capacompact} and Theorem \ref{capanoncompact}. One key ingredient are the symplectomorphisms constructed in the previous section and the other essential input comes from Hamiltonian circle actions. In the next subsection we quickly explain why such a circle action can be useful.

\subsection{Hamiltonian circle actions and capacities} 
A Hamiltonian $S^1$-manifold is a symplectic manifold $(M,\omega)$ that admits a Hamiltonian $H:M\to \R$ such that the associated Hamiltonian flow is 1-periodic. On these manifolds we have fairly good tools to understand the Gromov-width and the Hofer--Zehnder capacity. For example a lower bound for the Hofer--Zehnder capacity can often immediately be given in terms of the Hamiltonian $H$.

\begin{Lemma}\label{lem9}
Let $(M,\omega)$ be a compact symplectic manifold that admits a non-trivial semi-free Hamiltonian circle action with moment map 
$H:M\to \R$. Further assume that, if $M$ has a boundary, $H$ attains its minimum on the interior and its maximum constantly on the boundary of $M$. Then 
$$
c_{HZ}(M,\omega)\geq \mathrm{osc}(H)=\max H-\min H.
$$
\end{Lemma}
\begin{proof}
    We need to modify the Hamiltonian $H$ generating the circle action slightly so that it becomes admissible. This can be done with the help of a function
    $f:[a,b]\to [0,\infty)$ satisfying
    $$
    \begin{aligned}
    &0\leq f'(x)< 1, \\
    &f(x)=0\ \  \text{near}\ \ a,\\
    &f(x)=b-a-\varepsilon\ \  \text{near}\ \ b
    \end{aligned}
    $$
    with $a=\min H$ and $b=\max H$. Then all solutions to the Hamiltonian system with Hamiltonian $\tilde H=f\circ H$ have period 
    $$
    T=\frac{1}{f'(E)}> 1.
    $$
    Thus $\tilde H$ is admissible and we find the estimate 
    $$
    c_{HZ}(M,\omega)\geq \mathrm{osc}(\tilde H)=\mathrm{osc}(H)-\varepsilon,\ \ \ \ \forall \varepsilon>0
    $$
    and the claim follows.
    
\end{proof}
\noindent
In some cases also a lower bound of the Gromov-width can be obtained.
\begin{Proposition}[\cite{KT05}, Prop.\ 2.8]\label{prop7}
In addition to the assumptions of Lemma \ref{lem9} assume that the minimum is isolated, then
$$
c_G(M,\omega)\geq \mathrm{smin}(H)-\min H,
$$
where $\mathrm{smin}(H)$ denotes the second lowest critical value.
\end{Proposition}

\noindent 
On the other hand we are also in a good position to expect that some 1-point and 2-point Gromov-Witten invariants in suitable homology classes do not vanish, as we have pseudoholomorphic curves going through every point.
Indeed we may assume our compatible almost complex structure $J$ to be $S^1$-invariant\footnote{If $J$ is not invariant, we can always average the corresponding metric to be $S^1$-invariant and then redefine $J$.}. Now as described in \cite[Ex. 5.1.5]{DS17} the $S^1$-orbit of a gradient flow line of $H$ is a J-holomorphic sphere $u$ connecting critical points $c_\pm$ of $H$ and $\omega(u)=H(c_+)-H(c_-)$. Any non-critical point lies on a gradient flow line and at every critical point we have either incoming or outgoing flow lines. Thus we see that these gradient spheres go through every point. It is still highly non-trivial to show that some Gromov--Witten invariant does not vanish as there might be many (nodal or broken) pseudoholomorphic spheres through every point, so that they cancel in the count.
The first important result towards explicit computations of Gromov--Witten invariants in the context of Hamiltonian $S^1$-manifolds is the localization principle proved by McDuff and Tolman in \cite[sec.\ 4.2]{DT06}.\\
\ \\
\noindent
Our aim is to use Lu's Theorem \cite[Thm.\ 1.10]{Lu06} in the context of Hamiltonian circle actions. This set up was also studied by Hwang and Suh \cite{HS13} for closed Fano\footnote{They call $(M,\omega)$ Fano if there exists a compatible $S^1$-invariant almost complex structure such that all non-constant pseudo holomorphic spheres have positive Chern number. In particular monotone implies Fano.} symplectic manifolds with semi-free Hamiltonian circle action.

\begin{theorem}[Thm.\ 1.1.\ \cite{HS13}]\label{hs13}
    Let $(M,\omega)$ be a closed Fano symplectic manifold with a semifree Hamiltonian circle action. The Gromov width and the Hofer--Zehnder capacity are estimated as
    \begin{itemize}
        \item[(a)] $c_G(M,\omega) \leq \max(H) - \min (H) \leq c_{HZ}(M,\omega).$
        \item[(b)] Further if $H_{min}$ is a point, then
        $$
            c_G(M,\omega) = \mathrm{smin}(H)-\min(H),\ \ c_{HZ}(M,\omega)=\max(H)-\min(H).
        $$
    \end{itemize}
\end{theorem}
\noindent
One nice observation from this theorem is that it is compatible with taking products. If Fano symplectic manifolds $(M_1,\omega_1)$, $(M_2,\omega_2)$ with Hamiltonian circle actions generated by $H_1, H_2$ satisfy the prerequisites of Theorem \ref{hs13} (b), then so does $(M_1\times M_2, a\omega_1\oplus b\omega_2)$ with Hamiltonian $aH_1\circ\pi_1+bH_2\circ \pi_2$, where $\pi_1,\pi_2$ are the projections on the first resp.\ second factor. In particular 
$$
c_G(M_1\times M_2, a\omega_1\oplus b\omega_2) = \min\lbrace \vert a\vert c_G(M_1,\omega_1),\vert b\vert c_G(M_2,\omega_2)\rbrace
$$
and
$$
c_{HZ}(M_1\times M_2, a\omega_1\oplus b\omega_2)= \vert a\vert c_{HZ}(M_1, \omega_1)+\vert b\vert c_{HZ}(M_2,\omega_2)
$$
while for arbitrary symplectic manifolds only 
$$
c_{G}(M_1\times M_2, a\omega_1\oplus b\omega_2)\geq \min\lbrace \vert a\vert c_G(M_1,\omega_1),\vert b\vert c_G(M_2,\omega_2)\rbrace
$$
and
$$
c_{HZ}(M_1\times M_2, a\omega_1\oplus b\omega_2)\geq \vert a\vert c_{HZ}(M_1, \omega_1)+\vert b\vert c_{HZ}(M_2,\omega_2)
$$
holds.
\begin{Corollary}
    Let $(M,\omega)$ be a closed Fano symplectic manifold with a semifree Hamiltonian circle action and $H_{min}$ a point, then the Hofer--Zehnder capacity of any compact neighborhood of the zero-section in $(T^*M,\dd\lambda)$ is bounded.
\end{Corollary}
\begin{proof}
    The zero-section is a Lagrangian diffeomorphic to $M$. Also the diagonal in
    $(M\times M, \omega\ominus \omega)$ is such a Lagrangian. By the previous considerations 
    $c_{HZ}(M\times M, \omega\ominus \omega)$ is finite. This implies by Lagrangian neighborhood theorem that the Hofer--Zehnder capacity of some neighborhood of the zero-section must be finite. Scaling the fibers of the disc-bundle only scales the symplectic form and thus the capacity. We can therefore shrink any compact subset of $T^*M$ to fit in the neighborhood of the zero-section.
\end{proof}
\noindent
This corollary is a special case of the main theorem in \cite{AFO17} by Albers, Frauenfelder and Oancea. Indeed for all such Fano symplectic manifolds the Hurewicz map 
$$
\pi_2(M)\to H_2(M;\Z) 
$$
is nonzero, because all gradient spheres represent non-zero elements in $H_2(M;\Z)$.\\
\ \\
\noindent
\subsection{Hermitian symmetric spaces of compact type}
In this section we will compute the Gromov width and the Hofer--Zehnder capacity of any Hermitian symmetric space of compact type. The Gromov width for this class of symplectic manifolds was already computed by Loi, Mossa and Zuddas \cite{LMZ13}, but the proof we present is different and strongly relies on the existence of a semi-free Hamiltonian circle action. The Hofer--Zehnder capacity of Hermitian symmetric spaces is contained in the class of examples considered in \cite{ACC20}, we include the proof as we find it instructive to see why in the case of Hermitian symmetric spaces the lower bound and the upper bound given in \cite{ACC20} match. The Hamiltonian circle action and the capacity of $M\times M$ will be essential to prove Theorem \ref{capacompact}.\\
\ \\
\noindent
We want to use Theorem \ref{hs13}. Indeed Hermitian symmetric spaces are monotone and thus Fano as shown in \cite[Ch.\ 5, \S 16]{BH58}. Further as stated in the following lemma the representation of $M$ as adjoint orbit $O_Z\subset\mathfrak{g}$ almost immediately yields a Hamiltonian circle action. 
\begin{Lemma}\label{circle action}
Let $M$ be an irreducible Hermitian symmetric space. The Hamiltonian function
$$
\nu : M\cong O_Z\to \R,\ x\mapsto 2\pi B(Z,x)
$$
generates a semi-free circle action. Here $B(\cdot,\cdot):\mathfrak{g}\times\mathfrak{g}\to\R$ denotes the Killing form.
\end{Lemma}
\begin{proof}
Let us compute the Hamiltonian vector field,
$$
\dd\nu_x(\cdot)=2\pi B(Z,[x,\cdot])=-2\pi B([Z,\cdot],x)=2\pi\iota_{Z^\#}\sigma.
$$
We conclude $X_\nu=2\pi Z^\#$, which clearly generates a circle action, as the group generated by $Z$ is isomorphic to $S^1$. We shall see later that the prefactor is there to ensure that the period of the circle action is one (see figure \ref{fig6}).
\end{proof}
\begin{figure}[h]
	\centering
 \includegraphics[width=0.3\textwidth]{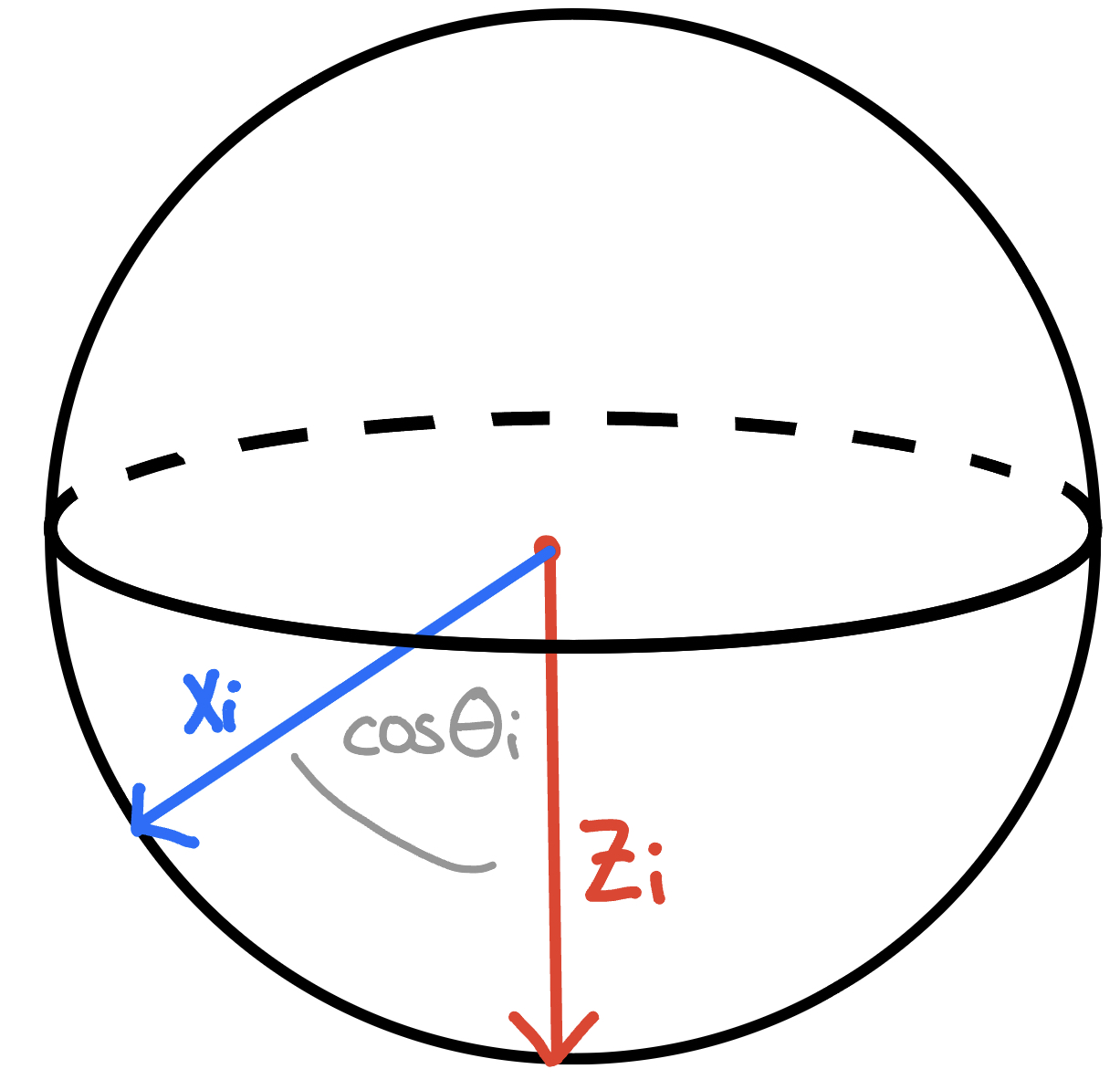}
	\caption{\textit{In order to satisfy $\mathrm{ad}_{Z_i}^2=-\mathrm{id}$, the norm of $Z_i$ needs to be one. We see that $(Z_i,x_i)$ is equal to the height function, which generates a circle action of period $2\pi$. In particular $\nu$ generates a circle action of period one.}}
    \label{fig6}
\end{figure}
\noindent
In order for $\nu$ to fulfill the prerequisites of Theorem \ref{hs13} we need to show that $\nu$ has an isolated minimum. Clearly $Z$ is a critical point of $\nu$. We claim that the point $Z$ is the isolated minimum of $\nu$.
\begin{Lemma}\label{lem12}
The Hessian of $\nu$ at $p=Z$ is positive definite, thus $p=Z$ is isolated and the global minimum.
\end{Lemma}
\begin{proof} Take $a,b\in \mathfrak{p}\cong T_ZM$, then
\begin{align*}
\mathrm{Hess}_Z(\nu)(a^\#_Z,b^\#_Z)&=a^\#\left (b^\#(\nu)\right)\big\vert_Z=a^\#\left(\dd\nu(b^\#)\right)\big\vert_Z=a^\#\left(2\pi B(Z,[b,p])\right)\big\vert_{p=Z}\\
&\left( \diff 2\pi B(Z,[b,\mathrm{Ad}_{e^{ta}}(p)])\right)\big\vert_{p=Z}=2\pi B(Z,[b,[a,Z]])=-2\pi B(a,b).
\end{align*}
We conclude that
$$
\mathrm{Hess}_Z=-2\pi B(\cdot,\cdot)\vert_{\mathfrak{p}\times\mathfrak{p}}.
$$
In particular the Hessian is positive definite as the Killing-form restricted to $\mathfrak{p}$ is negative definite in the compact case. This shows that $\nu$ is a local minimum. In general any Hamiltonian generating a circle action is a Morse-Bott function, its critical submanifolds are symplectic and their indices and coindices are even \cite[Lem. 5.5.7]{DS17}. By \cite[Lem. 5.5.5]{DS17} all level sets of such functions are connected, thus $p=Z$ is the global minimum.
\end{proof}
\noindent
This shows that our moment map $\nu$ satisfies the prerequisites of Theorem \ref{hs13}.
\begin{Lemma}\label{lem18}
    The Hamiltonian $\nu$ satisfies
    $$
    \max(\nu)-\min(\nu)=4\pi r,\ \ \mathrm{smin}-\min(\nu)=4\pi,
    $$
where $\mathrm{smin}(\nu)$ denotes the second lowest value of $\nu$ at a critical point.
\end{Lemma}
\begin{proof}
 Observe that for any $x\in M$ there exists a polysphere through $x$ and $Z$. From Proposition \ref{polysphere as suborbit} we know that the polyspheres are sub orbits and stay in an affine copy of $(\mathfrak{su}(2))^r$. Thus we can decompose $Z=\sum_iZ_i+c$ and $x=\sum_i x_i+c$, where $c$ the vector orthogonal to the affine subspace (see Figure \ref{fig5}). Recall that we picked $Z\in\mathfrak{g}$ the unique generator of the center of $K$ such that $[Z,[Z,v]]=-v$ for all $v\in \mathfrak{p}$. In particular it follows that $\vert Z_i\vert=1$. Indeed, this is equivalent to the above condition as $[Z_i,[Z_i,v]]=-\vert Z_i\vert^2 v$ for the case $\mathfrak{su}(2)$. See figure \ref{fig6} for a visualization.
 We compute
$$
(Z,x)=\sum_{i=1}^r(Z_i,x_i)+\vert c\vert^2=\sum_{i=1}^r\cos(\theta_i)+\vert c\vert^2,
$$
where we used that $\vert Z_i\vert=\vert x_i\vert =1$. Therefore we conclude
$$
\max(\nu)-\min(\nu)=4\pi r \ \ \ \text{and}\ \ \
\mathrm{smin}(\nu)-\min(\nu)=4\pi.
$$
\end{proof}
\begin{figure}
	\centering
 \includegraphics[width=0.5\textwidth]{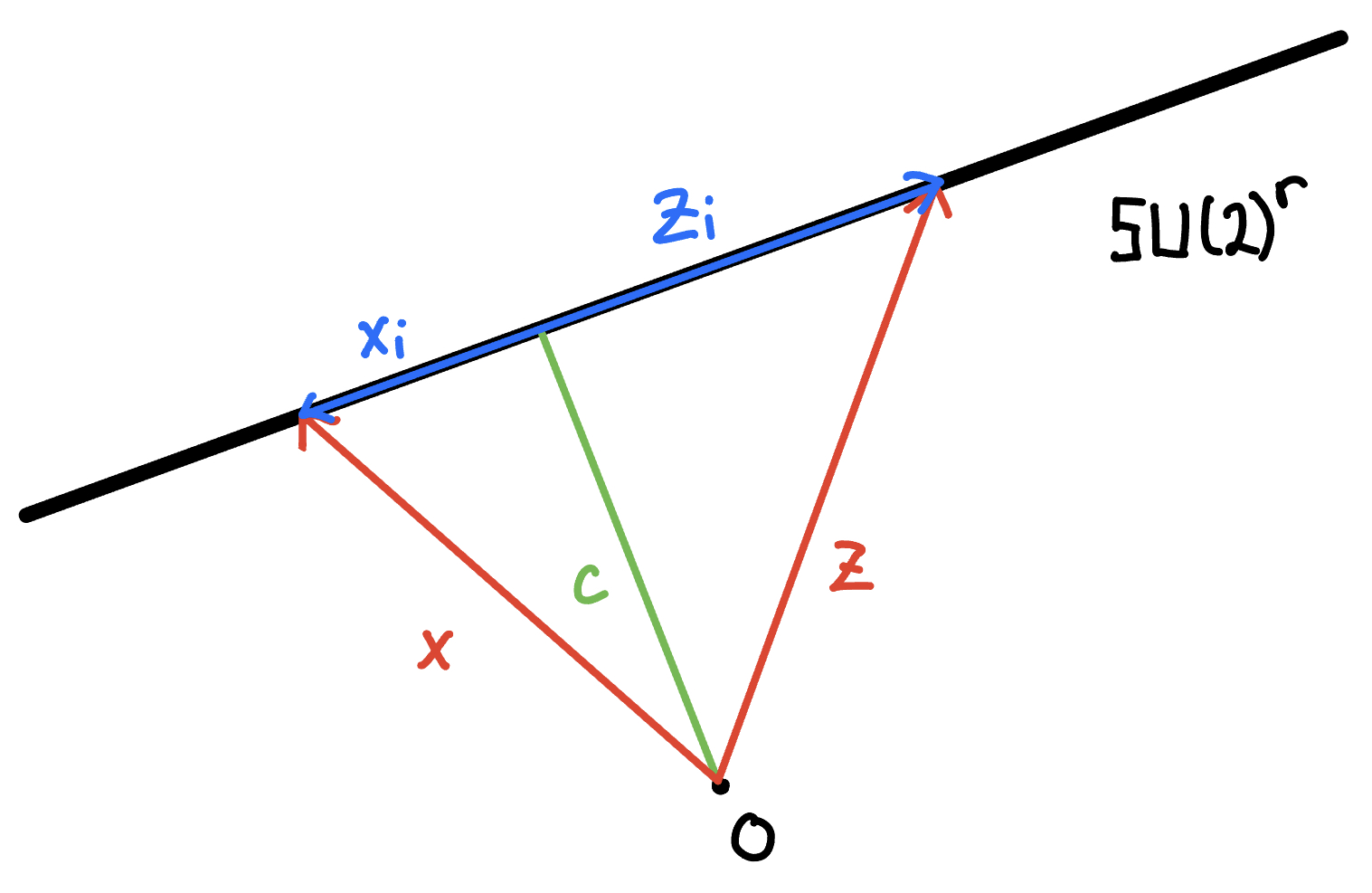}
	\caption{\textit{The figure schematically shows how the poly-spheres sits in an affine copy of $\mathfrak{su}(2)^r$.}}
    \label{fig5}
\end{figure}
\noindent
These three Lemmas together with Theorem \ref{hs13} prove the following theorems.

\begin{theorem}\label{thm13}
    Let $(M,g)$ be an irreducible Hermitian symmetric space of compact type. Denote by $r$ the rank of $M$ and normalize $\sigma$ such that $\sigma(A)=4\pi$ for $A$ the homology class of any factor in a polysphere. Then the Hofer--Zehnder capacity is given by
    $$
    c_{HZ}(M,\sigma)=4\pi r.
    $$
\end{theorem}

\begin{theorem}[\cite{LMZ13} Thm.\ 1]\label{cghss}
    Let $(M,g)$ be an irreducible Hermitian symmetric space of compact type. Normalize $\sigma$ such that $\sigma(A)=4\pi$ for $A$ the homology class of any factor in a polysphere. Then the Gromov width is given by
    $$
    c_{G}(M,\sigma)=4\pi.
    $$
\end{theorem}
\noindent

\subsection{Proof of Theorem \ref{capacompact}}
As an application of the diagonal embedding we compute the Gromov width and the Hofer--Zehnder capacity of $U$-neighborhoods of the zero-section in the tangent bundle of a Hermitian symmetric space of compact type. 
\begin{theorem} Normalize $\sigma$ such that evaluated on the generator of $H_2(M,Z)$ it takes value $4\pi$, then
$$
    c_{\mathrm G}(U_{2\sqrt{R}}M,\omega_{(1-R)\sigma})=c_{\mathrm{HZ}}(U_{2\sqrt{R}}M,\omega_{(1-R)\sigma})=c_{\mathrm{HZ}}^0(U_{2\sqrt{R}}M,\omega_{(1-R)\sigma})=\min \lbrace 1,R\rbrace 4\pi.
$$
\end{theorem}
\noindent
We prove the theorem by finding a lower bound for the Gromov width and an upper bound for the $\pi_1$-sensitive Hofer--Zehnder capacity. The theorem then follows as the inequalities 
$$
c_G\leq c_\mathrm{HZ}\leq c_\mathrm{HZ}^0
$$
are automatically satisfied. In order to find a lower bound we need to find symplectic embeddings of the standard ball into $(U_{2\sqrt{R}}M,\omega_{(1-R)\sigma})$. Following an idea of McDuff--Polterovich \cite{DP94} we will instead embedd a ball into $M\times M$ and use a holomorphic gradient flow to push the ball into the complement of $\bar\Delta$, hence into $U_{2\sqrt{R}}M$. Applying a Moser trick relative to $\bar \Delta$ we obtain the desired embedding.\\
\ \\
\noindent
Recall from Lemma \ref{circle action} that the Hamiltonian function
$$
\nu : M\cong O_Z\to \R,\ x\mapsto 2\pi B(Z,x)
$$
generates a semi-free circle action by holomorphic isometries. Here, $B(\cdot,\cdot):\mathfrak{g}\times\mathfrak{g}\to\R$ denotes the Killing form. 

\begin{Lemma}\label{hol}
    The gradient flow of $\nu$ is holomorphic.
\end{Lemma}
\begin{proof}
    The gradient flow of $\nu$ is holomorphic if and only if the Lie derivative $\mathcal{L}_{\nabla \nu}(j)=0$. Using the Leibniz rule for Lie derivatives we find 
    $$
    [\nabla\nu, Y]=\mathcal{L}_{\nabla \nu}(jY)=(\mathcal{L}_{\nabla \nu}j)(Y)+j(\mathcal{L}_{\nabla\nu}Y)=(\mathcal{L}_{\nabla\nu}j)(Y)+j[\nabla\nu,Y]).
    $$
    Hence, we need to prove that $[\nabla \nu,jY]=j[\nabla\nu,Y]$ for all vector fields $Y$ on $M$. As the tangent bundle is locally spanned by holomorphic vector fields we may assume $Y$ to be holomorphic, further $j\nabla\nu=-X_\nu$ is holomorphic as it generates an element of $G$. Now integrability of $j$ assures vanishing of Nijenhuis tensor and therefor
    $$
    [\nabla\nu, jY]=[jX_\nu, jY]=[X_\nu, Y]+j[jX_\nu, Y]+j[X_\nu, jY]=-[X_\nu, Y]=j^2[X_\nu, Y]=j[\nabla\nu, Y].
    $$
\end{proof}
\noindent
We can now use the holomorphic gradient flow of $\nu$ to push the symplectic balls found in Theorem \ref{cghss} into the complement of $\bar\Delta$.
\begin{Proposition} Set $A:=\min\lbrace 1,R\rbrace 4\pi$, then there is a symplectic embedding 
    $$
    (B(A-\varepsilon),\omega_0)\hookrightarrow (M\times M\setminus \bar\Delta,\sigma\ominus R\sigma)\ \forall \varepsilon>0,
    $$
    where $B(A)$ denotes the ball of capacity $A$.
\end{Proposition}
\begin{proof}
    By Theorem \ref{cghss} and and its proof we find a symplectic embedding 
    $$
     (B(A-\varepsilon),\omega_0)\hookrightarrow (M\times M, \sigma\ominus R\sigma)
    $$
    so that zero is mapped to the unique minimum of $\nu$, e.g.\ $(Z,Z)$, and $B(A)\subset \lbrace \nu < \mathrm{smin}(\nu)\rbrace$. As $(Z,Z)\not\in \bar\Delta$, the gradient flow $\varphi_t$ of the Hamiltonian $\nu$ can be used to push $\bar \Delta$ into the complement of the ball, i.e.
    $$
    \exists T>0\ \mathrm{s.t.}\ \varphi_T(\bar\Delta)\subset M\times M\setminus B(A-\varepsilon).
    $$
    Vice versa we obtain a symplectic embedding 
    $$
    (B(A-\varepsilon),\omega_0)\hookrightarrow (M\times M\setminus \varphi_T(\bar\Delta), \sigma\ominus R\sigma)\cong (M\times M\setminus\bar\Delta, \varphi_T^*(\sigma\ominus R\sigma)).
    $$
    By Lemma \ref{hol} the gradient flow of $\nu$ is holomorphic which implies that $j$ is compatible with $\varphi_t^*(\sigma\ominus R\sigma)$ for all $t$. Therefor $\bar\Delta$ is a finite union of closed symplectic submanifolds for all symplectic structures in the family $\varphi_T^*(\sigma\ominus R\sigma)$ and we can apply Moser's trick relative to $\bar \Delta$ to identify 
    $$
    (M\times M\setminus\bar\Delta,\ \varphi_T^*(\sigma\ominus R\sigma))\cong (M\times M\setminus\bar\Delta,\ \sigma\ominus R\sigma).
    $$
\end{proof}
\noindent
Letting $\varepsilon\to 0$ we conclude the following lower bound for the Gromov width.
\begin{Corollary} The Gromov width satisfies
    $$\min\lbrace 1,R\rbrace 4\pi\leq c_G(U_{2\sqrt{R}}M,\omega_{(1-R)\sigma}).$$
\end{Corollary}
\ \\
For the upper bound we use Lu's Theorem \cite[Thm.\ 1.10]{Lu06} or to be precise the following Corollary.

\begin{Corollary}[\cite{Bim23}, Cor.\ A.1]\label{upperbound}
 Let $(N,\omega)$ be a closed symplectic manifold of dimension $\dim N\geq 4$ and fix a closed\footnote{compact with no boundary!} connected submanifolds $D_\infty\subset N$ of codimension at least two. Denote by $[D_\infty]\in H_2(N,\Q)$ the induced homology classes. Suppose there exists a spherical homology class $A\in H_2(N;\Z)$ for which the Gromov-Witten invariant 
$$
\mathrm{GW}_{A,m+2}([pt.],[D_\infty], \beta_1, \ldots , \beta_m) \neq 0 
$$
for some homology classes $\beta_1,\ldots,\beta_m \in H_*(N;\Q)$ and an integer $m \geq 1$, then 
$$
c_{HZ}^\circ(N\setminus D_\infty,\omega)\leq \omega(A).
$$   
\end{Corollary}
\noindent
We want to apply this corollary to $N=M\times M$, $A$ the homology class of a factor in any polysphere and $D_\infty=\bar\Delta$. Observe that $\bar\Delta$ is not a closed connected submanifold, but connected and a finite union of smooth closed submanifolds. One can check, that the proof of Corollary \ref{upperbound} actually also goes through for our choice of $D_\infty=\bar\Delta$.\\
\ \\
We need one last ingredient to show that a Gromov--Witten invariant of the form 
$$\mathrm{GW}_{A,m+2}([pt.],[\bar\Delta], \beta_1, \ldots , \beta_m)$$ does not vanish.
\begin{Lemma}\label{intersection}
    The intersection product $[\bar\Delta]\cdot A$ does not vanish.
\end{Lemma}
\begin{proof}
    Observe that by Theorem \ref{diagonal} $(\sigma\ominus\sigma)\vert_{M\times M\setminus\bar\Delta}$ is exact. Now if there was a representative $S$ of $A$ that does not intersect $\bar\Delta$, then
    $$
    \int_S\sigma\ominus\sigma=0
    $$
    by Stoke's theorem. This is not true as $A$ is the homology class of a factor in a polysphere, hence has area $4\pi$.
\end{proof}
\noindent
We now proceed with the upper bound. By \cite[Lem.\ 15]{LMZ13} a Gromov--Witten invariant of the form $\mathrm{GW}_A([pt.],\alpha,\beta)$ does not vanish. Together with Lemma \ref{intersection} this implies
$$
\mathrm{GW}_{A,4}([pt.],[\bar\Delta], \alpha, \beta)=\mathrm{GW}_{A,3}([pt.], \alpha, \beta)\left ([\bar\Delta]\cdot A\right)\neq 0.
$$
Now an application of Theorem \ref{diagonal} and Corollary \ref{upperbound} yields the desired upper bound:
$$
c_{HZ}^0(U_{2\sqrt{R}}M,\omega_{(R-1)\sigma})=c_{HZ}^0(M\times M\setminus \bar\Delta,\sigma\ominus R\sigma)\leq (\sigma\ominus R\sigma)(A)= \min\lbrace 1,R\rbrace 4\pi.
$$
The minimum is obtained by choosing $A$ to be the homology class represented by a factor of a polysphere in the factor of $M\times M$ corresponding to $\min\lbrace 1,R\rbrace$.

\subsection{Proof of Theorem \ref{capanoncompact}}

The idea of proof of Theorem \ref{capanoncompact} is to find a lower bound by constructing a Hamiltonian circle action from the magnetic geodesic flow. The upper bound is obtained by first applying the symplectomorphisms of Theorem \ref{twisted to hyperkähler} and \ref{hyperkähler to constant} and then observing that the symplectic form $\dd\tau/2+\pi^*\sigma$ is precisely the symplectic structure for which Lu \cite[Thm. 1.3]{Lu06f} computed the Hofer--Zehnder capacity. In total we will compute the precise value of the Hofer--Zehnder capacity of sub level sets of the Hamiltonian generating the circle action, but only bounds for disk tangent bundles.

\begin{Proposition}
Let $(M,g,\sigma,j)$ be a Hermitian symmetric space of non-compact type and $s>1$, then the Hamiltonian 
$$
H:U_{1}M\to \R;\ (x,v)\mapsto g_x(h(jR_{jv,v})v,v),\ \  \text{where}\ \  h(y):=\frac{2\pi}{y}\left(\sqrt{s^2+y}-s\right)
$$
is well-defined, differentiable and generates a circle action.
\end{Proposition}
\begin{proof}
It is well defined, because the condition $s>1$ makes sure that the eigenvalues of $jR_{jv,v}$ are bounded from below by $-1$. Further $h$ is smooth, thus the same holds for $H$. Next we compute $\dd H$ using the chain rule. The horizontal part must vanish, as $g, j$ and $R$ are parallel, thus
\begin{align*}
    \dd H&=(0, g(\tilde h(jR_{jv,v})v,\cdot))
    =(\tilde h(jR_{jv,v})v,s\tilde  h(jR_{jv,v})jv)\begin{pmatrix} 
    -s\sigma & g\\
    g & 0
    \end{pmatrix},
\end{align*}
where $\tilde h(y)= h'(y)y+h(y)$ and we represent with respect to the splitting $TTM\cong \mathcal{H}\oplus\mathcal{V}$. It follows that
$$
X_H=(\tilde h(jR_{jv,v})v)^\mathcal{H}+s(\tilde h (jR_{jv,v})jv)^\mathcal{V}.
$$
The operator $jR_{jv,v}: T_xM\to T_xM$ restricts to the tangent space of the polydisc $\Sigma^r\cong(\ch^1)^r$ through $x$ tangent to $v$. As the polydiscs are totally geodesic $(\tilde h(jR_{jv,v})v)^\mathcal{H}\in TT\Sigma^r$. The embedding is also complex thus $(\tilde h(jR_{jv,v})jv)^\mathcal{V}\in TT\Sigma^r$.
In total we see that $X_H$ is tangent to $T\Sigma^r$. Observe that this means that the diagrams of the form
\[
\begin{tikzcd}
(D_{1}\Sigma)^r \arrow[dr,"\sum H\circ\pi_i"] \arrow[rr,"\dd\iota"] && U_{1}M\arrow[dl,"H"]  \\
 & \R 
\end{tikzcd}
\]
commute. On the polydisc $jR_{jv,v}$ is diagonal with respect to the product structure, i.e.\
$$
jR_{jv,v}= \begin{pmatrix} 
    - \vert v_1\vert ^2 & \dots  & 0\\
    \vdots & \ddots & \vdots\\
    0 & \dots  & - \vert v_r\vert^2
    \end{pmatrix}
$$
where $v_i$ is the projection of $v$ to the ith factor. Now it is easy to see that on the polydisc the Hamiltonian is given by
$$
H(x,v)=\sum_{i=1}^r\left (s-\sqrt{s^2-\vert v_i\vert^2}\right ).
$$
As shown in \cite[Lem.\ 2.1]{Bim23} this Hamiltonian generates a diagonal semi-free circle action on the product.  
\end{proof}

\begin{Lemma}\label{intertwine}
    Denote $\Psi: U_1M\to U_2M$ the composition of the symplectomorphism of Theorem \ref{twisted to hyperkähler} and Theorem \ref{hyperkähler to constant}, then
    $$
    2\pi E\circ\Psi=H: U_1M\to \R,
    $$
    where $E:TM\to \R; E(x,v)=\frac{1}{2}\vert v\vert^2$ is the kinetic Hamiltonian.
\end{Lemma}
\begin{proof}
    We compute
    \begin{align*}
        2\pi E(\Psi((x,v))&=\pi g_x(e^{a(jR_{jv,v})}v,e^{a(jR_{jv,v})}v)=\pi g_x(e^{2a(R_{jv,v})}v,v)\\
        &=g_x(h(R_{jv,v})v,v)=H(x,v).
    \end{align*} 
\end{proof}
\noindent
At the moment our base manifold $M\cong G/K$ is a Hermitian symmetric space of non-compact type, hence non-compact as a manifold. We take the quotient $\Gamma\setminus M$ with respect to a cocompact subgroup $\Gamma\subset K$. As the symplectomorphism $\Psi$ is equivariant with respect to the induced $G$-action, it survives a the quotient. Moreover, $H$ and $E$ still generate semi-free circle actions. From now on let $M=\Gamma\setminus G / K$ be a locally Hermitian with universal cover of non-compact type.
\begin{Lemma}\label{capasub}
    Let $\rho<2$ and $s>1$, then
    $$
    c_{HZ}(\lbrace H< \pi \rho^2\rbrace,\omega_{s\sigma})=c^0_{HZ}(\lbrace H<\pi \rho^2\rbrace,\omega_{s\sigma})=\pi \rho^2.
    $$
\end{Lemma}
\begin{proof}
Lemma \ref{intertwine} shows that 
$$
\Psi\left(H^{-1}(\pi \rho^2)\right)=E^{-1}\left(\rho^2/2\right)=D_\rho M.
$$
     We can compactify the disc bundle $(D_\rho M,\dd\tau/2+s\pi^*\sigma)$ using a Lerman cut \cite{ler} with respect to the Hamiltonian circle action induced by $E$. The resulting symplectic manifold $(\overline{D_\rho M},\omega)$ is a symplectic fibration 
    $$
    \cp^n\hookrightarrow \overline{D_\rho M}\twoheadrightarrow M.
    $$
    As the universal cover of $M$ is of non-compact type, M is aspherical. In particular the fiber class $[\cp^1]\in H_2(\overline{D_\rho M},\Z)$ is minimal. We can apply a theorem by Hofer--Viterbo \cite[Thm.\ 1.16]{HV92} to see that $\omega([\cp^1])=\pi\rho^2$ is an upper bound. The lower bound is automatic as $E$ generates a semi-free crircle action (see Lemma \ref{lem9}).
\end{proof}
\noindent
We can now prove Theorem \ref{capanoncompact}.

\begin{Theorem}[Thm.\ \ref{capanoncompact}]
Let $(M,g,\sigma)$ be isometrically covered by an irreducible Hermitian symmetric space of non-compact type with rank $r$, then
$$
2\pi\left(s-\sqrt{s^2-1/r}\right)\leq c_{HZ}(D_1 M,\omega_{s\sigma})\leq c_{HZ}^0(D_1 M,\omega_{s\sigma})\leq 2\pi r\left (s-\sqrt{s^2- 1}\right )
$$
for any constant $s>1$.
\end{Theorem}
\begin{proof}
The bounds for $c_{HZ}(D_\rho M,\omega_s)$ are now obtained from Lemma \ref{capasub} by asking what is the largest sub level set of $H$ that lies in $D_1 M$ and what is the smallest sub level set of $H$ that contains $D_1 M$.\\
\ \\
\noindent
Assume $(x,v)\in \left \lbrace H\leq 2\pi\left(s-\sqrt{s^2-1/r}-s\right)\right \rbrace $, then we can use that $v$ splits into $\sum_i v_i$ along a polydisc and we obtain the following chain of inequalities
\begin{align*}
     &H(x,v)=g_x(h(jR_{jv,v})v,v)\leq 2\pi \left(s-\sqrt{s^2-1/r}\right)\\
     \Rightarrow\ \ &\sum_i2\pi \left(s-\sqrt{s^2- \vert v_i\vert ^2}\right)\leq 2\pi \left(\sqrt{s^2-1/r}-s\right)\\
    \Rightarrow\ \ &\left(s-\sqrt{s^2-\vert v_i\vert ^2}\right)\leq \left(s-\sqrt{s^2-1/r}\right)\ \ \ \ \ \forall\ i\\
    \Rightarrow\ \ & \vert v_i\vert^2\leq 1/r\ \ \ \ \ \forall\ i\\
    \Rightarrow\ \ & \vert v\vert^2\leq 1.
\end{align*}
This implies that the sub level set 
$\left\lbrace H\leq 2\pi\left(s-\sqrt{s^2-1/r}\right)\right\rbrace$ is in $ D_1 M$ and thus the lower bound holds.
For the upper bound of the capacity we have to find a sub level set of $H$ that contains $D_1 M$. For this let $(x,v)\in D_1 M$, then again considering the splitting of $v$ into $\sum_i v_i$ along a polydisc we find
$$
\vert v\vert ^2\leq 1
\Rightarrow\ \ \vert v_i\vert \leq 1 \ \ \ \forall i.
$$
Then by monotonicity of the square root it follows that
\begin{align*}
 &\left(s-\sqrt{s^2-\vert v_i\vert ^2}\right)\leq \left(s-\sqrt{s^2-1}\right)\ \ \ \forall i\\
\Rightarrow\ \ &H(x,v)=g_x(h(jR_{jv,v})v,v)=\sum_i 2\pi \left(s-\sqrt{s^2- \vert v_i\vert ^2}\right)\leq 2\pi r\left(s-\sqrt{s^2-1}\right).
\end{align*}
This implies that $D_1 M\subset  H^{-1}\left(2\pi r\left(s- \sqrt{s^2-1}\right)\right)$ and the upper bound follows, which finishes the proof.
\end{proof}

\bibliographystyle{abbrv}
\bibliography{ref}
\end{document}